\documentclass[12pt]{amsart}
\usepackage{amsmath, amssymb, amsthm, mathrsfs, textcomp, colonequals,
  comment,soul}
\usepackage{epic}
\usepackage[shortlabels]{enumitem}
\usepackage{xypic}
\usepackage[backref]{hyperref}
\usepackage[alphabetic,backrefs,lite]{amsrefs}
\usepackage[usenames,dvipsnames]{color}
\usepackage{fullpage}

\hypersetup{
  colorlinks   = true, %Colours links instead of ugly boxes
  urlcolor     = blue, %Colour for external hyperlinks
  linkcolor    = blue, %Colour of internal links
  citecolor   = blue %Colour of citations
}

\numberwithin{equation}{section}
\allowdisplaybreaks[1]
\newtheorem{theorem}[equation]{Theorem}

\newtheorem{corollary}[equation]{Corollary}
\newtheorem{prop}[equation]{Proposition}
\newtheorem{lemma}[equation]{Lemma}

\theoremstyle{definition}

\newtheorem{remark}[equation]{Remark}
\newtheorem{notation}[equation]{Notation}
\newtheorem{defn}[equation]{Definition}

\newtheorem{example}[equation]{Example}

\setenumerate{itemsep=2pt}

\newcommand{\padma}[1]{{\color{Magenta} \sf $\clubsuit\clubsuit\clubsuit$ Padma: [#1]}}

\newcommand{\ol}[1]{\overline{#1}}
\newcommand{\mc}[1]{\mathcal{#1}}
\newcommand{\mf}[1]{\mathfrak{#1}}

\newcommand{\rats}{\mathbb Q}
\newcommand{\Z}{\mathbb Z}
\newcommand{\ints}{\mathbb Z}

\renewcommand{\O}{\mathcal{O}}

\newcommand{\aff}{\mathbb A}
\newcommand{\FF}{\mathbb F}
\renewcommand{\P}{\mathbb P}
\newcommand{\proj}{\mathbb P}

\renewcommand{\phi}{\varphi}

\newcommand{\Spec}{\mathrm{Spec} \ }

\DeclareMathOperator{\Gal}{Gal}

\DeclareMathOperator{\divi}{div}

\DeclareMathOperator{\reg}{reg}

\DeclareMathOperator{\lcm}{lcm}

\DeclareMathOperator{\Proj}{Proj}

\title{Explicit minimal embedded resolutions of divisors on models of the projective line}

\author{Andrew Obus}
\address{Baruch College}
\curraddr{1 Bernard Baruch Way. New York, NY 10010, USA}
\email{andrewobus@gmail.com}
\thanks{The first author was supported by the National Science
  Foundation under DMS Grants No. 1602054, 1900396, and 2047638.  He was also
  supported by a grant from the Simons Foundation (\#706038: AO), as
  well as a PSC-CUNY Award, jointly funded by The Professional Staff
Congress and The City University of New York.}

\author{Padmavathi Srinivasan}
\address{University of Georgia}
\curraddr{452 Boyd Graduate Studies, 1023 D. W. Brooks Drive, Athens, GA 30602.}
\email{Padmavathi.Srinivasan@uga.edu}

\subjclass[2010]{Primary: 14B05, 14J17; Secondary: 13F30, 14H25}
\keywords{Mac Lane valaution, embedded resolution, regular model}

\date{\today}

\begin{document}

\maketitle

\begin{abstract}
  Let $K$ be a discretely valued field with ring of integers $\mc{O}_K$ with perfect residue field. Let $K(x)$ be the rational function field in one variable. Let $\P^1_{\mc{O}_K}$ be the standard smooth model of $\P^1_K$ with coordinate $x$. Let $f(x) \in \mc{O}_K[x]$ be a squarefree polynomial with corresponding divisor of zeroes $\divi_0(f)$ on $\P^1_{\mc{O}_K}$.
  % A minimal embedded resolution of the pair $(\P^1_{\mc{O}_K}, \divi_0(f))$ is a regular model $\mc{Y}$ of $\P^1_K$ with a birational morphism $\pi \colon \mc{Y} \rightarrow \P^1_{\mc{O}_K}$ such that the strict transform of $\divi_0(f)$ is regular, and such that any other modification $\pi' \colon \mc{Y}' \rightarrow \P^1_{\mc{O}_K}$ with $\mc{Y}'$ regular and the strict transform of $\divi_0(f)$ regular factors uniquely as $\mc{Y}' \rightarrow \mc{Y} \xrightarrow{\pi} \P^1_{\mc{O}_K}$.
  We give an explicit description of the minimal embedded resolution $\mc{Y}$ of the pair $(\P^1_{\mc{O}_K}, \divi_0(f))$ by using Mac Lane's theory to write down the discrete valuations on $K(x)$ corresponding to the irreducible components of the special fiber of $\mc{Y}$.  
\end{abstract}

%\tableofcontents

\section{Introduction}\label{Sintro}
Let $K$ be a discretely valued field with ring of integers $\mc{O}_K$ with perfect residue field. Let $K(x)$ be the rational function field in one variable. Let $\P^1_{\mc{O}_K}$ be the standard smooth model of $\P^1_K$ with coordinate $x$. Let $f(x) \in \mc{O}_K[x]$  with corresponding divisor of zeroes $\divi_0(f)$ on $\P^1_{\mc{O}_K}$. A minimal embedded resolution of the pair $(\P^1_{\mc{O}_K}, \divi_0(f))$ is a regular model $\mc{Y}$ of $\P^1_K$ with a birational morphism $\pi \colon \mc{Y} \rightarrow \P^1_{\mc{O}_K}$ such that the strict transform of $\divi_0(f)$ is regular, and such that any other modification $\pi' \colon \mc{Y}' \rightarrow \P^1_{\mc{O}_K}$ with $\mc{Y}'$ regular and the strict transform of $\divi_0(f)$ regular factors uniquely as $\mc{Y}' \rightarrow \mc{Y} \xrightarrow{\pi} \P^1_{\mc{O}_K}$.\footnote{Note that such a resolution exists only when $f$ is squarefree.} The main result of this paper is the following theorem (See Theorem~\ref{Thorizontalregular} for a more precise statement, with notation as defined in Notation~\ref{Nyfprime}.  Also see the last paragraph of \S\ref{Soutline}.)

%\padma{Er, I haven't stated or displayed the main theorem in any form, good for abstract, baaaad for intro..}

%\padma{See Theorem~\ref{Thorizontalregular} for a more precise statement.
\begin{theorem}\label{Tintro}
Let $f \in \mc{O}_K[x]$ be a squarefree polynomial. There is an explicit description of the minimal embedded resolution $\mc{Y}$ of the pair $(\P^1_{\mc{O}_K}, \divi_0(f))$ when $\deg(f) \geq 2$\footnote{When $\deg(f) = 1$, the divisor $\divi_0(f)$ is already regular on the standard model $\P^1_{\mc{O}_K}$.}. More specifically, we write down the discrete valuations on $K(x)$ corresponding to the irreducible components of the special fiber of $\mc{Y}$.  
%One can explicitly construct a regular model $\mc{Y}$ of $\proj^1_K$ such that the closure in $\mc{Y}$ of the zero locus of $f$ on $\proj^1_K$ is regular.  Furthermore, one can bound the number of irreducible components of the special fiber of $\mc{Y}$ from above.
\end{theorem}
%}

\begin{remark}
This minimal embedded resolution is a key technical input to \cite{OS1}, where it is used to help prove a conductor-discriminant inequality for hyperelliptic curves in residue characteristic $\neq 2$, as we now describe. 
\end{remark}

It is well-known that an algorithm for strong embedded resolution of singularities in dimension $n-1$ gives rise to an algorithm for resolution of singularities in dimension $n$. 
%\padma{CITE \cite[Section~2.3]{Kollar} or something for surfaces over fields (this doesn't say any more than what I just did though, and doesn't state the arithmetic version)!! or trace back references in Liu-Lorenzini}
%\padma{MORE PRECISE: It is well-known (CAN I CITE LIU LORENZINI OR SOMETHING?) that if $X \rightarrow Y$ is a finite branched cover of curves of degree prime to the residue characteristic, and $\mc{Y}$ is a regular model of $Y$, an explicit embedded resolution of the closure of the branch divisor of $X \rightarrow Y$ in $\mc{Y}$ gives rise to an explicit regular model of $\mc{X}$.}  
The motivation for the current paper is to explicitly understand regular models of cyclic covers of $\P^1_K$ branched at $\divi_0(f)$ by explicitly constructing embedded resolutions of pairs $(\P^1_{\mc{O}_K}, \divi_0(f))$ first. The eventual goal of these constructions is to give an upper bound on the number of components in the exceptional fiber of such a resolution; see \cite{OS1} for an application to proving conductor-discriminant inequalities for degree $2$ covers of $\P^1_K$, and forthcoming work of the authors for higher degree cyclic covers. We do so by capitalizing on the recent revival in \cite{Ruth, ObusWewers} of explicit descriptions of normal and regular models of $\P^1_K$, using descriptions of valuations of $K(x)$ (now called ``Mac Lane valuations'')  going back to Mac Lane \cite{MacLane}. 
%, to and in turn use these descriptions to bound the number of components in an 

In \cite[Proposition~3.4]{Ruth}, R\"uth shows that normal models of $\P^1_K$ are in bijection with non-empty finite collections of discrete valuations on $K(x)$ (extending the given valuation on $K$) whose residue fields have transcendence degree $1$ over the residue field of $K$. 
%\padma{Add connector sentence, to say henceforth we will simply describe such collections of valuations on $K(x)$.} 
Over algebraically closed fields, it is known that analogous valuations with value group $\rats$ on the rational function field can be constructed from supremum norms on non-archimedean disks. Over non-algebraically closed discretely valued fields, R\"uth   (\cite[Proposition~4.56]{Ruth}, restated in Proposition~\ref{Pvaldiskoid}) shows that there is a similar description of valuations in terms of ``diskoids'', which are Galois stable collections of non-archimedean disks defined over the algebraic closure. In fact, he shows that these diskoids can be explicitly described by giving a certain sequence of polynomials $\phi_i$ in $K[x]$ of increasing degree (whose roots correspond to the ``centers'' of a nested sequence of diskoids) and a corresponding sequence of rational numbers $\lambda_i$ (``radii'' of the diskoids) -- such a description goes back to Mac Lane \cite{MacLane} from 1936. These $\phi_i$ can be thought of as successive lower degree approximations to the roots of a polynomial $f \in \mc{O}_K[x]$, and each rational number $\lambda_i$ is simply $\nu_K(\phi_i(\alpha))$ for any root $\alpha$ of $f$ (Corollary~\ref{Cpseudoeval}). Using successive $\phi_i$-adic expansions, one can easily compute the valuation of any given polynomial from this description, by a procedure analogous to the computation of the Gauss valuation; see the discussion surrounding (\ref{E:augval}).
%There is a simple algorithm to compute the valuation of any given polynomial
%analogous to the computation of the Gauss valuation -- one simply computes successive $\phi_i$-adic expansions of the given polynomial, and uses $\lambda_i$ for the valuation of $\phi_i$ and recursively evaluate valuation of coefficients in the expansions (which are lower degree polynomials) . 
Mac Lane valuations have been implemented in Sage in \cite{RuthSage}. In \cite[Theorem~7.8]{ObusWewers} (restated here in Proposition~\ref{Pgeneralresolution}), the authors describe the minimal regular resolution of a model of $\P^1_K$ with irreducible special fiber corresponding to a valuation $v$, using the same polynomials $\phi_i$ that show up in the description of $v$, and natural Farey paths between successive $\lambda_i$.   
% \padma{Is it worth writing a few extra sentences in the intro about diskoids just to give some intiution for where this ``proper key polynomial'' business comes from? I am not sure how much new nonstandard words to bring in to the intro. From giving talks, it seems to help folks who are used to the non-archimedean perspective like Joe/Daniel etc. but not really Dino etc. Not sure who exactly we should be addressing this intro. to... It's what helps me when I think about these...
% 
% 
% I guess the people who will probably referee this are some valuation theory folks?}

The bulk of the paper is devoted to proving Theorem~\ref{Thorizontalregular}, which is Theorem~\ref{Tintro} in the case where $f$ is monic and irreducible and the residue field $k$ is algebraically closed.  The general result can easily be derived from Theorem~\ref{Thorizontalregular}; see Remarks~\ref{Rreducible} and \ref{Retaledescent}.  So for the rest of the introduction, assume $f$ is monic and irreducible and $k$ is algebraically closed.  To each such $f \in \O_K[x]$, there is a canonical diskoid centered about the roots of $f$ giving rise to a valuation $v_f$ on $K(x)$ (\S\ref{Sregularmodelprelims}). By R\"uth's correspondence, this valuation $v_f$ corresponds to a normal model of $\P^1_K$ with irreducible special fiber, which we will call the $v_f$-model. By $v_f$-component, we mean the strict transform of the special fiber of the $v_f$-model in any model that dominates it. In what follows, we use $\divi_0(f)$ to mean the zero divisor of $f$ on any model of $\proj^1_K$; the model will be clear from context.

Concurrent to our work in \cite{OS1old} (an earlier version of \cite{OS1}), 
%\padma{should we say this paper was spun off from our earlier preprint here?}  
in \cite[Theorem~3.16]{KW}, the authors also noted that $\divi_0(f)$ is a normal crossings divisor on the minimal regular resolution $\mc{Y}_{v_f}^{\reg}$ of the $v_f$-model $\mc{Y}_{v_f}$, which implies that the minimal regular model $\mc{Y}_{v_f,0}^{\reg}$ dominating $\mc{Y}_{v_f}^{\reg}$ and $\proj^1_{\mc{O}_K}$ is an embedded resolution of $(\proj^1_{\mc{O}_K}, \divi_0(f))$. However, $\mc{Y}_{v_f,0}^{\reg}$ is never the \textit{minimal} embedded resolution of the pair $(\proj^1_{\mc{O}_K}, \divi_0(f))$. In fact, for the applications to regular models of hyperelliptic curves, we are sometimes forced to work with (regular) contractions of $\mc{Y}_{v_f,0}^{\reg}$ where the strict transform of $\divi_0(f)$ is also regular. Determining whether the horizontal part of $\divi_0(f)$ remains regular on these contractions can be challenging because it might specialize to a node.
% ; worse yet, $\divi_0(f)$ may no longer be normal crossings on these contractions.

The two main insights of this paper are the following. First, using the machinery of Mac Lane valuations, it is possible to explicitly modify $f$ to write down a rational function $g$ that cuts out the unique irreducible \textit{horizontal} divisor in $\divi_0(f)$ on natural contractions of $\mc{Y}_{v_f,0}^{\reg}$. Note that checking regularity of $\divi_0(g)$ at its unique closed point $y$ is equivalent to checking whether $g$ is in the square of the maximal ideal at $y$. This is hard to check directly since this local ring is $2$-dimensional. The second main insight is to use the $\phi_n$-adic\footnote{Here $\phi_n$ is the last polynomial that shows up in the Mac Lane description of $v_f$.} expansion of $f$ to write down an analogous explicit decomposition $g = \sum_i g_i$. The terms $g_i$ in this decomposition vanish along \textit{vertical} components through the closed point $y$ (a computation back in a $1$-dimensional local ring), even though $g$ itself does not, and we can exploit the orders of vanishing to determine when $g$ is in the square of the maximal ideal.
%We can use this decomposition to understand when the individual $g_i$ are in the square of the maximal ideal. This strategy succeeds since most of the $g_i$ actually vanish along \textit{vertical} components through $y$ even though $g$ does not.  
It turns out that the Mac Lane descriptions of vertical components are tailor-made for computing orders of vanishing of functions along these components!

%  \padma{ ---- Old sentence: Second, checking whether $g$ is in the square of the maximal ideal at its unique closed point $y$ (which is tantamount to checking regularity of $\divi_0(g)$ at $y$) can be reduced to computing the order of vanishing of $g$ along the \textit{vertical} components passing through $y$ -- computing these orders of vanishing is exactly what the Mac Lane descriptions of these vertical components is built for!  ----
%  
%  I'm afraid this is confusing. The vertical order of vanishing of $g$ is $0$ because $\divi_0(g)$ is horizontal! Technically in Lemma~\ref{Lindividualterms}, we do not ever compute vertical orders of vanishing for $g$ directly, but for terms in the expansion of $g$. Is there a way to succintly say this?}
%\padma{this isn't correct as stated for Type I or Type II models, must be fixed. There is no statement of a lemma or a proposition in the paper that computes the order of vanishing of $g$ along vertical components in the Type I and Type II cases (we do have such a Lemma for Type III). Implicitly in the proof of Lemma~\ref{Lindividualterms}, we compute orders of vanishing of individual terms in the expansion of $g$ along specific vertical components and combine with the CES lemma.}

%The key technical arguments in this paper (Section~\ref{Shorizontal}) needed for analyzing regularity require a finer analysis of valuations of individual terms in the $\phi$-adic expansion of $f$. 
Our main theorem shows that quite often it is possible to contract entire tails in the dual graph of $\mc{Y}_{v_f,0}^{\reg}$ and in fact, the minimal embedded resolution we are after is the minimal regular resolution of one of two neighbouring components of the $v_f$-component in the dual graph of $\mc{Y}_{v_f,0}^{\reg}$. We do not see any way to deduce our main theorem directly from \cite[Theorem~3.16]{KW}.

\subsection{Outline of the paper}\label{Soutline}

 In \S\ref{Smaclane}, we introduce Mac Lane valuations.  As we have
mentioned, a normal model of $\proj^1_K$ corresponds to a finite set of Mac Lane
valuations, one valuation for each irreducible component of the
special fiber.  Mac Lane valuations are also in one-to-one
correspondence with \emph{diskoids}, which are Galois orbits of
rigid-analytic disks in $\proj^1_{\ol{K}}$.  We will use the diskoid
perspective often, and it is introduced in 
in \S\ref{Sdiskoids}.

In \S\ref{Smaclanemodels}, we prove several results about the
correspondence between Mac Lane valuations and normal models of $\proj^1_K$.  For
instance, if $\mc{Y}$ is a normal model of $\proj^1_K$ with special
fiber consisting of several irreducible components, each corresponding
to a Mac Lane valuation, results in \S\ref{Smaclanemodels} can be used
to determine which irreducible component a point of $\proj^1_K$
specializes to.  After this, we cite a result (Proposition~\ref{Pgeneralresolution}) from \cite{ObusWewers}
giving an explicit criterion for when a normal model of $\proj^1_K$ is
regular.  More specifically, using that Mac Lane valuations correspond to
normal models of $\proj^1_K$ with irreducible special fiber,
Proposition~\ref{Pgeneralresolution} takes a Mac Lane valuation as
input and gives the minimal regular resolution of the corresponding
normal model as output (as a finite set of Mac Lane valuations, of course)!

In \S\ref{Shorizontal}, we first define the canonical valuation $v_f$ associated to a polynomial $f$. The minimal embedded resolution of the pair $(\P^1_{\mc{O}_K}, \divi_0(f))$ is a certain contraction of $\mc{Y}_{v_f,0}^{\reg}$. So we are lead to an analysis of regularity of the strict transform of $\divi_0(f)$, which we will henceforth call $D$, on natural contractions of $\mc{Y}_{v_f,0}^{\reg}$. To this end, in \S\ref{Shorizontal} we first define three types of regular models of $\P^1_K$ that can arise as contractions of $\mc{Y}_{v_f,0}^{\reg}$. Viewing these contractions as a sequence of closed point blow-downs, a short argument shows that if we want the blow-down to stay regular and dominate $\proj^1_{\mc{O}_K}$, there is a unique component that can be blown down at every stage (for instance, the $v_f$-component is the only $-1$-component that can be blown down in the model $\mc{Y}_{v_f,0}^{\reg}$ by the \emph{minimality} of the construction of $\mc{Y}_{v_f, 0}^{\reg}$). As we proceed through this natural sequence of blow-downs, we first go through a sequence of models we call ``Type I'' models. If $D$ stays regular on all Type I regular blow-downs of $\mc{Y}_{v_f,0}^{\reg}$, we then move on to the ``Type II'' models. We continue contracting in this way, and after the Type II models, naturally comes the unique ``Type III'' model. (See Definition~\ref{D3types}.) 

In \S\ref{Sproof}, we run thus argument.  The crux is to show that $D$ is \textit{not} regular on the unique Type III model (Proposition~\ref{Pmodeltoosmall}), and we use this to show that the minimal embedded resolution of $D$ must be a special Type I or a Type II model (Corollary~\ref{CIorII}). We then show that if $D$ is regular on a Type I or Type II model, then the model must include a component corresponding to one of two additional canonical valuations attached to the polynomial $f$, denoted $v_f'$, $v^{''}_f$ (Proposition~\ref{Pregular}) -- these turn out to be neighbouring valuations to $v_f$ in the dual graph of $\mc{Y}_{v_f,0}^{\reg}$. The technical lemmas needed for these regularity arguments use an analysis of valuations of individual terms in the $\phi_n$-adic expansion of $f$ along \textit{vertical} components of these models (Lemma~\ref{LtypeIII} for the unique Type III model, and Lemma~\ref{Lindividualterms} for Type I and Type II models). The Mac Lane machinery for describing these vertical components is perfectly equipped for carrying out such calculations. Finally, in Theorem~\ref{Thorizontalregular}, we show that the minimal embedded resolution of the pair $(\P^1_{\mc{O}_K}, \divi_0(f))$ is the minimal regular model dominating $\proj^1_{\mc{O}_K}$ and either the $v_f'$-model or the $v^{''}_f$-model.

\section*{Notation and conventions}
Throughout, $K$ is a Henselian field with respect to a
discrete valuation $\nu_K$.   In much of the paper (\S\ref{Salgclosed}, \S\ref{Smaclanemodels}, \S\ref{Shorizontal}, and all of \S\ref{Sproof} until the very end) we will further assume
that the residue field $k$ of $K$ is \emph{algebraically closed}, but this will be noted specifically and is not a running assumption for the paper.
We denote an algebraic closure of $K$ by $\ol{K}$.
%All algebraic
%extensions of $K$ are assumed to live inside $\ol{K}$.
%This means that
%for any algebraic extension $L/K$, there is a preferred embedding
%$\iota_L \in \Hom_K(L, \ol{K})$, namely the inclusion.
We fix a uniformizer $\pi_K$ of $\nu_K$ and normalize $\nu_K$ so that $\nu_K(\pi_K) = 1$. Note that the valuation $\nu_K$ uniquely extends to a valuation on $\ol{K}$, which we also call $\nu_K$.

% \padma{CAN THIS BE REPLACED WITH THE PREVIOUS SENTENCE FOR THIS PAPER? -- If $L/K$ is an algebraic extension, the
% valuation $\nu_K$ extends uniquely to $L$, and we write $\mc{O}_L$
% for the valuation ring of $L$.  If $L/K$ is a
% \emph{finite} algebraic extension, we write $\nu_L$ for
% the renormalization of the extension of $\nu_K$ to $L$ such that
% $\nu_L(\pi_L) = 1$, where $\pi_L$ is a uniformizer of $L$.  In fact,
% we conflate $\nu_K$ and $\nu_L$ with their extensions to $\ol{K}$, and
% furthermore with the restrictions of these extended valuations to subextensions of
% $\ol{K}/K$. So, for instance, if $L$ and $M$ are finite extensions of
% $K$ and $m \in M$, it makes sense to evaluate $\nu_L(m)$, even if $L$
% and $M$ are linearly disjoint over $K$.  We simply have $\nu_L(m) =
% [L:K]\nu_K(m) = [L:K][M:K]^{-1}\nu_M(m)$. }

For an integral $K$-scheme or $\mc{O}_K$-scheme $S$, we denote the corresponding function
field by $K(S)$. If $\mc{Y} \rightarrow \mc{O}_K$ is an arithmetic
surface, an irreducible codimension 1 subscheme of $\mc{Y}$ is called
\emph{vertical} if it lies in a fiber of $\mc{Y} \to \mc{O}_K$, and
\emph{horizontal} otherwise. Let $f \in K(\mathcal{Y})$. We denote the
divisor of zeroes of $f$ by $\divi_0(f)$. 
% \padma{Will go away: If $\divi(f) = \sum_i m_i
% \Gamma_i$, call a component $\Gamma_i$ for which $m_i$ is odd an
% \emph{odd component} of $\divi(f)$ on $\mathcal{Y}$. Similarly define
% \emph{even component} of $\divi(f)$ (this includes every component
% $\Gamma_i$ for which $m_i = 0$).} 
For any discrete valuation $v$, we denote the corresponding value group by $\Gamma_v$.
% If $E$ is a regular codimension $1$ point of $\mc{Y}$, we denote the corresponding discrete valuation on the function field of $\mc{Y}$ by $\nu_E$.
If $P$ is a closed point on $\mc{Y}$, we denote the corresponding local ring by $\mc{O}_{\mc{Y},P}$ and maximal ideal by $\mathfrak{m}_{\mc{Y},P}$. 

Throughout this paper, we fix a system of homogeneous coordinates $\P^1_{K} =
\Proj K[x_0,x_1]$, and $x \colonequals x_1/x_0$ and $\P^1_{\O_K}
\colonequals \Proj \mc{O}_K[x_0,x_1]$.  

All minimal polynomials are assumed to be monic.  
%The \emph{$K$-degree of an element $\alpha \in \ol{K}$} is the degree of its minimal polynomial over $K$.  
When we refer to the \emph{denominator} of a rational
number, we mean the positive denominator when the rational number is
expressed as a reduced fraction.

\section*{Acknowledgements}
The authors would like to acknowledge the hospitality of the
Mathematisches Forschungsinstitut Oberwolfach, where they participated in the
``Research in Pairs'' program that was integral to the writing of this paper.  They would also like to thank Dino Lorenzini for useful conversations, and the referees for their thoughtful comments and suggestions to improve the exposition.

\section*{Data availability statement}
Data sharing not applicable to this article as no datasets were generated or analysed during the current study.

\section{Mac Lane valuations}\label{Smaclane}

\subsection{Definitions and facts}\label{Sbasicmaclane}

We recall the theory of inductive valuations,
which was first developed by Mac Lane in \cite{MacLane}.  We also use
the more recent \cite{Ruth} as a reference.  Inductive valuations give us an
explicit way to talk about normal models of $\proj^1$.

Define a \emph{geometric valuation} of
$K(x)$ to be a discrete valuation that restricts to $\nu_K$ on $K$ and
whose residue field is a finitely generated extension of $k$ with
transcendence degree $1$.   We place a partial order $\preceq$ on valuations by defining $v
\preceq w$ if $v(f) \leq w(f)$ for all $f \in K[x]$.  Let $v_0$ be the
\emph{Gauss valuation} on $K(x)$.  This is defined on $K[x]$ by
$v_0(a_0 + a_1x + \cdots a_nx^n) = \min_{0 \leq i \leq n}\nu_K(a_i)$,
and then extended to $K(x)$.

We consider geometric valuations $v$ such that $v \succeq v_0$. By the triangle inequality, these are precisely those geometric valuations for which
$v(x) \geq 0$.  This entails no loss of generality, since $x$ can
always be replaced by $x^{-1}$. We would like an explicit formula for describing geometric valuations, similar to the formula above for the Gauss valuation, and this is achieved by the so-called
\emph{inductive valuations} or \emph{Mac Lane valuations}. Observe that the Gauss valuation is described using the $x$-adic expansion of a polynomial. The idea of a Mac Lane valuation is to ``declare'' certain polynomials $\phi_i$ to have higher valuation than expected, and then to compute the valuation recursively using $\phi_i$-adic expansions.

More specifically, if $v$ is a geometric valuation such that $v
\succeq v_0$, the concept of a \emph{key polynomial} over $v$ is
defined in \cite[Definition 4.1]{MacLane} (or \cite[Definition 4.7]{Ruth}).  Key polynomials are monic polynomials
in $\mc{O}_K[x]$ --- we do not give a definition, which would require more
terminology than we need to develop, but see Lemmas \ref{Lfdegreebasic} and \ref{Lfdegree} below for the most useful properties.  If $\phi \in \mc{O}_K[x]$ is a
key polynomial over $v$, then for $\lambda > v(\phi)$, 
we define an \emph{augmented valuation} $v' = [v, v'(\phi) = \lambda]$ on $K[x]$ by 
\begin{equation}\label{E:augval} v'(a_0 + a_1\phi + \cdots + a_r\phi^r) = \min_{0 \leq i \leq r}
v(a_i) + i\lambda \end{equation} whenever the $a_i \in K[x]$ are polynomials with
degree less than $\deg(\phi)$.  We should think of this as a ``base
$\phi$ expansion'', and of $v'(f)$ as being the minimum valuation of a
term in the base $\phi$ expansion of $f$ when the valuation of $\phi$ is
declared to be $\lambda$.  By \cite[Theorems 4.2, 5.1]{MacLane} (see also
 \cite[Lemmas 4.11, 4.17]{Ruth}), $v'$ is in fact a discrete
 valuation.  In fact, the key polynomials are more or less the polynomials
 $\phi$ for which the construction above yields a discrete valuation
 for $\lambda > v(\phi)$.
 The valuation $v'$ extends to $K(x)$. 

We extend this notation to write Mac Lane valuations in the following
form: $$[v_0, v_1(\phi_1(x)) = \lambda_1, \ldots, v_n(\phi_n(x)) = \lambda_n].$$
Here each $\phi_i(x) \in \mc{O}_K[x]$ is a key polynomial over $v_{i-1}$, we
have that $\deg(\phi_{i-1}(x)) \mid \deg(\phi_i(x))$, and each $\lambda_i$
satisfies $\lambda_i > v_{i-1}(\phi_i(x))$.  By abuse of notation,
we refer to such a valuation as $v_n$ (if we have not given it another
name), and we identify $v_{i}$ with $[v_0, v_1(\phi_1(x)) = \lambda_1, \ldots,
v_{i}(\phi_{i}(x)) = \lambda_{i}]$ for each $i \leq n$.  The valuation
$v_i$ is called a \emph{truncation} of $v_n$.  One sees without much
difficulty that $v_n(\phi_i) = \lambda_i$ for all $i$ between $1$ and
$n$.  

It turns out that the set of Mac Lane valuations on $K(x)$ exactly
coincides with the set of geometric valuations $v$ with $v \succeq
v_0$ (\cite[Corollary 7.4]{FGMN} and \cite[Theorem 8.1]{MacLane}, or \cite[Theorem
4.31]{Ruth}). Furthermore, 
every Mac Lane valuation is equal to one where the degrees of the
$\phi_i$ are strictly increasing (\cite[Lemma 15.1]{MacLane} or \cite[Remark 4.16]{Ruth}), so we may and do assume this to be the case for the rest of the paper.
This has the consequence
that the number $n$ is well-defined.  We call $n$ the
\emph{inductive valuation length} of $v$.  In fact, by \cite[Lemma
15.3]{MacLane} (or \cite[Lemma
4.33]{Ruth}), the degrees of the $\phi_i$ and the values of the $\lambda_i$ are invariants of $v$, once
we require that they be strictly increasing.  If $f$ is a key polynomial
over $v = [v_0,\, v_1(\phi_1) = \lambda_1, \ldots,\, v_n(\phi_n) =
\lambda_n]$ and either $\deg(f) > \deg(\phi_n)$ or $v = v_0$, we call $f$ a \emph{proper
  key polynomial over $v$}.  By our convention, each $\phi_i$ is a proper key
polynomial over $v_{i-1}$.  

We collect some basic results on Mac Lane valuations and key
polynomials that will be used repeatedly.  
\begin{lemma}\label{Lfdegreebasic}
Suppose $f$ is a proper key polynomial over $v = [v_0,\, v_1(\phi_1) = \lambda_1,
\ldots,\, v_n(\phi_n) = \lambda_n]$ with $n \geq 1$. If $f = \phi_n^e + a_{e-1}\phi_n^{e-1} + \cdots
  + a_0$ is the $\phi_n$-adic expansion of $f$, then $v_n(a_0) =
  v_n(\phi_n^e) = e\lambda_n$, and $v_n(a_i\phi_n^i) \geq e\lambda_n$
  for all $i \in \{1, \ldots, e-1\}$. In particular, $v_n(f) = e\lambda_n$.
\end{lemma}

\begin{proof}
This follows from \cite[Theorem 9.4]{MacLane} (or \cite[Lemma 4.19(ii),
(iii)]{Ruth}).
\end{proof}

\begin{example}\label{Ebasickey}
If $K = \text{Frac}(W(\ol{\mathbb{F}}_3))$, then the polynomial $f(x) = x^3 - 9$ is a
proper key polynomial over $[v_0,\, v_1(x) = 2/3]$.  In
accordance with Lemma~\ref{Lfdegreebasic}, we have $v_1(f) = v_1(9) = v_1(x^3) =
3 \cdot 2/3 = 2$.  If we extend $v_1$ to a valuation $[v_0,\, v_1(x) =
2/3,\, v_2(f(x)) = \lambda_2]$ with $\lambda_2 > 2$, then the
valuation $v_2$ notices ``cancellation'' in $x^3 - 9$ that $v_1$ does not.
\end{example}

\subsection{Mac Lane valuations and diskoids}\label{Sdiskoids}
Given $\phi \in \mc{O}_K[x]$ monic, irreducible and $\lambda \in \rats_{\geq 0}$, we define
the \emph{diskoid} $D(\phi, \lambda)$ with ``center'' $\phi$ and radius $\lambda$ to
be $D(\phi, \lambda) \colonequals \{\alpha \in \ol{K} \mid \nu_K(\phi(\alpha)) \geq
\lambda\}$ (we only treat diskoids with \emph{non-negative, finite} radius in the
sense of \cite[Definition 4.40]{Ruth}). By \cite[Lemma~4.43]{Ruth}, a
diskoid is a union of a disk with all of its $\Gal(\overline{K}/K)$-conjugates.
Such a diskoid is said to be \emph{defined} over $K$, since $\phi \in \mc{O}_K[x]$.  Notice that the \emph{larger} $\lambda$ is, the \emph{smaller} the diskoid is.  We now state the fundamental
correspondence between Mac Lane valuations and diskoids.

\begin{prop}[{cf.\ \cite[Theorem 4.56]{Ruth}, see also
  \cite[Proposition 5.4]{ObusWewers}}]\label{Pvaldiskoid}
There is a bijection from the set of diskoids to the set of Mac Lane valuations that
sends a diskoid $D$ to the valuation $v_D$ defined by
$v_D(f) = \inf_{\alpha \in D} \nu_K(f(\alpha))$.  The inverse sends a
Mac Lane valuation $v = [v_0,\,\ldots, \, v_n(\phi_n) = \lambda_n]$ to
the diskoid $D_v$ defined by $D_v = D(\phi_n, \lambda_n)$.  Alternatively, 
$$D_v = \{\alpha \in \ol{K} \ \mid \ \nu_K(f(\alpha)) \geq v(f) \ \forall f \in K[x] \},$$
is a presentation of $D_v$ independent of the description of $v$ as a Mac Lane valuation. 

Lastly, if $D$ and $D'$ are
diskoids, then $D \subseteq D'$ if and only if $v_D \succeq v_{D'}$.  If $v$
and $v'$ are Mac Lane valuations, then $v \succeq v'$ if and only if
$D_v \subseteq D_{v'}$.
\end{prop}

The following proposition is crucial for our method.

\begin{prop}\label{Pbestlowerapprox}
Let $\alpha \in \mc{O}_{\ol{K}}$, and let $f \in K[x]$ be the
minimal polynomial for $\alpha$.  Then there exists a unique Mac Lane valuation $v_f = [v_0,\, \ldots,\, v_n(\phi_n) = \lambda_n]$ over which $f$ is a proper key polynomial. 
\end{prop}

\begin{proof}
Consider the unique valuation $v_L$ on $L := K[x]/(f)$
extending $v_K$.  This lifts to a discrete \emph{pseudovaluation} on $K[x]$ in the language of
\cite[\S4.6]{Ruth} (a valuation which can take the value $\infty$ on
an ideal, in this case $(f)$).  By \cite[Corollary 4.67]{Ruth}, it can
be written as a so-called ``infinite inductive valuation'' $[v_0,\, \ldots,\, v_n(\phi_n) =
  \lambda_n, v_{n+1}(f) = \infty]$, with $f$ a proper key polynomial over
  $v_f := [v_0,\, \ldots,\, v_n(\phi_n) =
  \lambda_n]$.  This shows the existence of $v_f$.   If $f$ is a proper key
  polynomial over some other valuation $v$, then for sufficiently
  large $\lambda$, one can construct inductive
  valuations $v_f' = [v_f,\, v_f'(f) = \lambda]$ and $v' = [v,\, v'(f) = \lambda]$.
  By Proposition~\ref{Pvaldiskoid}, these inductive valuations
  correspond to the same diskoid, and are thus the same.  Applying the
  ``only if'' direction of \cite[Theorem 4.33]{Ruth} (or \cite[Theorem
  15.3]{MacLane}) to $v_f'$ and
  $v'$, and then the ``if'' direction of the same theorem to $v_f$ and
  $v$ shows that $v_f = v$.
\end{proof}

To close out \S\ref{Sdiskoids}, we prove several results linking Mac Lane
valuations evaluated at a polynomial to the valuation of that polynomial at a
particular point.

\begin{defn}[{\cite[Definition 4.4, Lemma 4.24]{Ruth}}]\label{Dmaclanereciprocal}
If $v = [v_0,\, v_1(\phi_1) = \lambda_1, \ldots,\, v_n(\phi_n) =
\lambda_n]$ is a Mac Lane valuation and $f \in K[x]$, then a
\emph{$v$-reciprocal} of $f$ is a polynomial $f' \in K[x]$ such that $v(ff'
- 1) > 0$ and $v(f') = v_{n-1}(f') = -v(f)$.  
\end{defn}

By \cite[Lemma 9.1]{MacLane} (or \cite[Lemma 4.24]{Ruth}), any $f \in K[x]$ with $v(f) = v_{n-1}(f)$
has a $v$-reciprocal.  In this case, it is clear from Definition \ref{Dmaclanereciprocal} that $f$ and $f'$
being $v$-reciprocals is a symmetric relation.

\begin{prop}\label{Pvalsame}
Suppose $v = [v_0,\, v_1(\phi_1) = \lambda_1, \ldots,\, v_n(\phi_n) =
\lambda_n]$ is a Mac Lane valuation, $\alpha \in D(\phi_n,
\lambda_n)$, and $g \in K[x]$ such that $v(g) = v_{n-1}(g)$.  Then $\nu_K(g(\alpha)) = v(g)$.  
\end{prop}

\begin{proof}
Let $D \colonequals D(\phi_n, \lambda_n)$ be the diskoid corresponding to $v$ and let $D' \colonequals D(g, \nu_K(g(\alpha)))$ with corresponding valuation $v'$.
These two diskoids share the common element
$\alpha$.  By \cite[Lemma 4.44]{Ruth}, either $D \subseteq D'$ or $D'
\subseteq D$, and then Proposition \ref{Pvaldiskoid} shows that
either $v' \preceq v$ or $v \preceq v'$.

Since $\alpha \in D$, by Proposition \ref{Pvaldiskoid} we have
$\nu_K(g(\alpha)) \geq v(g)$.
Suppose $\nu_K(g(\alpha)) > v(g)$.
Since $v'(g) = \nu_K(g(\alpha))$ by definition, we have $v(g) <
v'(g)$. Since either $v' \preceq v$ or $v \preceq v'$, it follows that $v \preceq v'$. 
Let $g' \in K[x]$ be a $v$-reciprocal of $g$, i.e., $gg' = 1 +
h$ with $v(h) > 0$ ($g'$ exists because $v(g) = v_{n-1}(g)$).  Since
$v \preceq v'$, we have $0 < v(h) \leq v'(h)$.  In particular,
$v'(gg') = v(gg') = 0$, so $v'(g') = -v'(g) < -v(g) = v(g')$.  But this
contradicts $v \preceq v'$.
\end{proof}

\begin{corollary}\label{Cpseudoeval}
If $f$ is a key polynomial over $v = [v_0,\, v_1(\phi_1) = \lambda_1,
\ldots,\, v_n(\phi_n) = \lambda_n]$ with root $\alpha \in \ol{K}$,
then $\nu_K(g(\alpha)) = v(g)$ for all $g \in \mc{O}_K[x]$ of degree less
than $\deg(f)$.  In particular, $\nu_K(\phi_i(\alpha))
= \lambda_i$ for all $1 \leq i \leq n$.
\end{corollary}

\begin{proof}
Consider a Mac Lane valuation 
$w_f = [v_0,\, v_1(\phi_1) = \lambda_1,\, \ldots,\, v_n(\phi_n) =
  \lambda_n,\, v_{n+1}(f) = \lambda_{n+1}]$, with $\lambda_{n+1}$
  large.  Then $v_{n+1}(g) = v_n(g)$ and $\alpha
  \in D(f, \lambda_{n+1})$, so the corollary follows from Proposition \ref{Pvalsame}. 
\end{proof}

%\begin{corollary}\label{Cnoroot}
%Let $v = [v_0,\, v_1(\phi_1) = \lambda_1, \ldots,\, v_n(\phi_n) =
%\lambda_n]$ be a Mac Lane valuation.  If $a_e\phi_n^e +
%a_{e-1}\phi_n^{e-1} + \cdots + a_0$ is
%the $\phi_n$-adic expansion of a polynomial $f \in K[x]$ and if $v(f)
%= v(a_0)$, then $f$ has no roots $\alpha$ such
%that $\nu_K(\phi_n(\alpha)) > \lambda_n$.
%\end{corollary}
%
%\begin{proof}
%Observe that since $v(f) = v(a_0)$, we have $v(a_0) \leq
%v(a_i\phi_n^i)$ for all $i > 0$.  Also, by the definition of $v$, we have $v(a_i) = v_{n-1}(a_i)$ for all $0 \leq i
%\leq e$.   If $\alpha$ is a root of $f$ such that
%$\nu_K(\phi_n(\alpha)) > \lambda_n = v(\phi_n)$, then $\alpha \in
%D(\phi_n,\lambda_n)$. Combining the last two sentences with
%Proposition \ref{Pvalsame} gives $\nu_K(a_i(\alpha)) = v(a_i)$ for $i
%\geq 0$.  So 
%$\nu_K(a_i(\alpha)\phi_n(\alpha)^i) > v(a_i\phi_n^i) \geq v(a_0) =
%v(a_0(\alpha))$ for $i > 0$.  Since $a_0(\alpha)$ is the unique term
%of the $\phi_n$-adic expansion of $f(\alpha)$ with lowest valuation,
%$\alpha$ cannot be a root of $f$. 
%\end{proof}
\subsection{Ramification of Mac Lane valuations}\label{Salgclosed}

For \S\ref{Salgclosed}, we assume that the residue field $k$ of $K$ is \emph{algebraically closed}.

If $v$ and $w$ are two Mac Lane valuations such that the value group $\Gamma_w$
contains the value group $\Gamma_{v}$, we write $e(w/v)$ for the
ramification index $[\Gamma_w : \Gamma_v]$. 

\begin{remark}\label{Rrelram}
Observe that if $[v_0,\, v_1(\phi_1) = \lambda_1, \ldots,\, v_n(\phi_n) =
\lambda_n]$ is a Mac Lane valuation, where each $\lambda_i = b_i/c_i$
in lowest terms, then the ramification index $e(v_n/v_0)$ equals $\lcm(c_1, \ldots, c_n)$.
Consequently, $e(v_i/v_j) = \lcm(c_1, \ldots, c_i)/\lcm(c_1,\ldots,
c_j)$ for $i \geq j$. 
\end{remark}

\begin{lemma}\label{Lfdegree}
Suppose $f$ is a proper key polynomial over $v = [v_0,\, v_1(\phi_1) = \lambda_1,
\ldots,\, v_n(\phi_n) = \lambda_n]$.
\begin{enumerate}[\upshape (i)]
\item If $n = 0$, then $f$ is linear.  If $n \geq 1$, then $\phi_1$ is
  linear.  Every monic linear polynomial in $\mc{O}_K[x]$ is a key polynomial over $v_0$.
\item If $n \geq 1$, then $\deg(f)/\deg(\phi_n) = e(v_n/v_{n-1})$.
\end{enumerate}
\end{lemma}

\begin{proof}
Part (i) follows from \cite[Remark~5.2(i)]{ObusWewers} for $n=0$, and
then for general $n \geq 1$ by applying the $n=0$ case to $\phi_1$ and $v_0$.
Part (ii) follows from \cite[Theorem~12.1]{MacLane} (one can also use the second equation of \cite[Lemma~4.30]{Ruth},
where $\FF_m = \FF_{m-1} = k$, but note that \cite[Lemma~4.30]{Ruth} is
incorrect as stated --- the expression $e(v_m/v_{m-1})$ should be
replaced by $e(v_{m-1}/v_{m-2})$).
\end{proof}

\begin{remark}
  The assumption $k$ algebraically closed is required above to apply \cite[Remark~5.2(i)]{ObusWewers} and to assume $\FF_m = \FF_{m-1} = k$ in \cite[Lemma~4.30]{Ruth}.
\end{remark}

\begin{corollary}\label{CvalueofN}
Let $v = [v_0,\, v_1(\phi_1) = \lambda_1,
\ldots,\, v_n(\phi_n) = \lambda_n]$ be a Mac Lane valuation of
inductive valuation length $n \geq 1$.  Write
$\lambda_i = b_i/c_i$ in lowest terms for all $i$.  
Let $N_n = \lcm_{i < n} c_i$ if $n > 1$, and
let $N_n= 1$ if $n  =1$.  Then $N_n = e(v_{n-1}/v_0) = \deg(\phi_n)$,
and thus $\Gamma_{v_{n-1}} = (1/N_n)\ints = (1/\deg(\phi_n))\ints$.
\end{corollary}

\begin{proof}
That $\deg(\phi_1) = 1$ is Lemma \ref{Lfdegree}(i), which
proves the corollary when $n = 1$.  By Remark~\ref{Rrelram}, $e(v_{j+1}/v_j)\lcm(c_1, \ldots, c_j) = \lcm(c_1, \ldots,
c_{j+1})$. The rest of the corollary follows from Lemma \ref{Lfdegree}(ii) and induction. 
\end{proof}

\begin{lemma}\label{Lvnvprime}
Let $[v_0,\, v_1(\phi_1) = \lambda_1,\ldots,\, v_n(\phi_n) =
\lambda_n]$ be a valuation over which $f$ is a proper key polynomial. Then for $1 \leq i \leq n$, we have
$\lambda_i \notin \Gamma_{v_{i-1}} = (1/N_i)\ints$.
\end{lemma}
%\padma{Would it be better to state this lemma immediately after Corollary~\ref{CvalueofN}? That way we can also get rid of the ``assumption`` $\lambda_n \notin \Gamma_{v_{n-1}}$ at the start of Section~\ref{Scontractions}.}
\begin{proof}
If $\lambda_i \in \Gamma_{v_{i-1}}$, then $e(v_i/v_{i-1}) = 1$.  If $i =
n$, applying Lemma~\ref{Lfdegree}(ii) to $v_n$, contradicts the fact that $\deg(f) >
\deg(\phi_n)$.  For $i < n$, applying Lemma~\ref{Lfdegree}(ii) to
$v_i$ contradicts the fact that $\deg(\phi_{i+1}) > \deg(\phi_i)$.
\end{proof}

\section{Mac Lane valuations, normal models and regular resolutions}\label{Smaclanemodels}
  In \S\ref{Shorizontaldivs}, we prove results on the specialization of
  horizontal divisors, expressed in terms of Mac Lane
  valuations.
  In \S\ref{Sresolution} we recall a result from \cite{ObusWewers},
  giving a criterion in terms of Mac Lane valuations for when a 
  model of $\proj^1_K$ is regular.  Lastly, in
  \S\ref{Scontractions}, we discuss valuations that are in a geometric
  sense ``nearby'' to a given Mac Lane valuation in a regular model of
  $\proj^1_K$.  These valuations will play a crucial role throughout
  the rest of the paper.
  
A \emph{normal model} of $\proj^1_K$ is a flat, normal, proper
$\mc{O}_K$-curve with generic fiber isomorphic to $\proj^1_K$.
By \cite[Corollary 3.18]{Ruth}, normal models $\mc{Y}$ of $\proj^1_K$ are in one-to-one correspondence with
non-empty finite collections of geometric valuations, by sending
$\mc{Y}$ to the collection of geometric valuations corresponding to
the local rings at the generic points of the irreducible components of
the special fiber of $\mc{Y}$.
Via this correspondence, the multiplicity of an irreducible
component of the special fiber of a normal model $\mc{Y}$ of $\proj^1_K$
corresponding to a Mac Lane valuation $v$ equals $e(v/v_0)$.

We say that a normal model of $\proj^1_K$ \emph{includes} a Mac Lane
valuation $v$ if a component of the special fiber corresponds to
$v$. If $\mc{Y}$ includes $v$, we call the corresponding irreducible component of its
special fiber the \emph{$v$-component} of the special fiber of $\mc{Y}$ (or simply
the $v$-component of $\mc{Y}$, even though it is not an irreducible
component of $\mc{Y}$).
If $S$ is a finite set of Mac Lane valuations, then the
\emph{$S$-model of $\proj^1_K$} is the normal model including exactly
the valuations in $S$.  If $S = \{v\}$, we simply say the $v$-model
instead of the $\{v\}$-model.  
Recall that we fixed a coordinate $x$ on $\proj^1_K$, that is, a rational function $x$ on $\proj^1_K$ such that
$K(\proj^1_K) = K(x)$.

For the remainder of \S\ref{Smaclanemodels}, we assume the residue field $k$ of $K$ is \emph{algebraically closed}, but the statements above about the correspondence between normal models and collections of geometric valuations are true without this assumption.  

\subsection{Specialization of horizontal divisors}\label{Shorizontaldivs}
Each $\alpha \in \ol{K} \cup \{\infty\}$ corresponds to a point of
$\proj^1(\ol{K})$ given by $x = \alpha$, which lies over a unique
closed point of $\proj^1_K$.  If
$\mc{Y}$ is a normal model of $\proj^1_K$, the closure of this point in
$\mc{Y}$ is a subscheme that we call $D_{\alpha}$; note that $D_{\alpha}$ is a horizontal divisor (the model will be clear from context, so we omit it to lighten the notation).

If $v$ is a Mac Lane valuation, then the reduced special fiber of the
$v$-model of $\proj^1_K$ is isomorphic to $\proj^1_k$ (see, e.g.,
\cite[Lemma~7.1]{ObusWewers}).  It will be useful to have
an explicit coordinate on this special fiber (that is, a rational function $y$
such that the function field of the special fiber is $k(y)$). 
\begin{lemma}\label{LA1coordinate}
Let $v = [v_0,\, v_1(\phi_1) = \lambda_1,\ldots,\, v_n(\phi_n) =
\lambda_n]$ be a Mac Lane valuation, and let $e = e(v_n/v_{n-1})$.
There exists a monomial $t$ in $\phi_1,\ldots,
\phi_{n-1}$ such that $v(t \phi_n^e) = 0$, and for any such $t$, the restriction of
$t\phi_n^e$ to the reduced special fiber of the $v$-model of
$\proj^1_K$ is a coordinate on the $v$-component that vanishes at the specialization of $\phi_n = 0$.
\end{lemma}

\begin{proof}
  Let $\mc{O} \subseteq K[x]$ be the subring of elements $f$ such that
$v(f) \geq 0$, and let $\mc{O}^+$ be the ideal of elements $g$ where
$v(g) > 0$.  Let $e = e(v_n/v_{n-1})$.
By \cite[Theorem 12.1]{MacLane} (or \cite[Lemma 4.29]{Ruth} and the discussion before that 
lemma), $\mc{O}/\mc{O}^+ \cong k[y]$, where $y$ is the image of $t\phi_n^e$ in
$\mc{O}/\mc{O}^+$, for any $t \in K[x]$ with $v(t\phi_n^e) =
0$ and $v(t) = v_{n-1}(t)$ (in the notation of \cite{Ruth}, the
example used is $t =
(S')^{\ell}$).  Since $v(\phi_n^e) \in
\Gamma_{v_{n-1}}$, we can take $t$ to be a monomial in $\phi_1,
\ldots, \phi_{n-1}$.  Since
$\Spec \mc{O}$ is an affine open of the $v$-model with reduced special
fiber $\Spec \mc{O}/\mc{O}^+ \cong \Spec k[y] \cong \aff^1_k \subseteq
\proj^1_k$, we have that $y$ is a coordinate on the
reduced special fiber of the $v$-model of $\proj^1_K$.  
\end{proof} 

\begin{prop}\label{Psmallerdiskoid}
Let $v = [v_0,\, v_1(\phi_1) = \lambda_1, \ldots,\, v_n(\phi_n) =
\lambda_n]$ be a Mac Lane valuation and let $\mc{Y}$ be the 
$v$-model of $\proj^1_K$.  As $\alpha$ ranges over $\ol{K}$, all
$D_{\alpha}$ with $\nu_K(\phi_n(\alpha)) > \lambda_n$ meet on the special
fiber, all $D_{\alpha}$ with $\nu_K(\phi_n(\alpha)) < \lambda_n$ meet
at a different point
on the special fiber, and no $D_{\alpha}$ with $\nu_K(\phi_n(\alpha))
\neq \lambda_n$ meets any $D_{\beta}$ with $\nu_K(\phi_n(\beta)) = \lambda_n$.
\end{prop}

\begin{proof}
Let $\mc{Y}$ be the $v$-model of $\proj^1_K$. Using the coordinate $y \colonequals t\phi_n^e$ from Lemma~\ref{LA1coordinate} on the reduced special fiber of $\mc{Y}$, we will show that 
  all $\alpha \in \ol{K}$ with $\nu_K(\phi_n(\alpha)) < \lambda_n$
  specialize to $y=\infty$, all $\alpha \in \ol{K}$ with
  $\nu_K(\phi_n(\alpha)) > \lambda_n$
  specialize to $y=0$ and all $\alpha \in \ol{K}$ with
  $\nu_K(\phi_n(\alpha)) = \lambda_n$
  specialize to some point $y = a$ with $a \notin \{0,\infty\}$. We now work out the details.

Let $\mc{O} \subseteq K[x]$ be the subring of elements $f$ such that
$v(f) \geq 0$, and let $\mc{O}^+$ be the ideal of elements $g$ where
$v(g) > 0$. Suppose $\alpha \in D(\phi_n, \lambda_n)$.  Proposition~\ref{Pvaldiskoid} shows that $\nu_K(g(\alpha)) > 0$ for $g
\in \mc{O}^+$, thus evaluating $y$ at
$\alpha$ gives a well-defined element of $k$. Furthermore, $y=y(\alpha)$ is precisely the point where $D_{\alpha}$ meets the special fiber of $\mc{Y}$. We now compute:
\begin{align*}
y(\alpha) = 0 &\Leftrightarrow \nu_K(t(\alpha)\phi_n(\alpha)^e) > 0
  \\
&\Leftrightarrow \nu_K(t(\alpha)\phi_n(\alpha)^e) >
                  v(t\phi_n^e) \\
&\Leftrightarrow \nu_K(\phi_n(\alpha)) > \lambda_n \quad \quad \quad \quad \quad \quad \quad (\because \ \nu_K(t(\alpha)) = v(t)).
\end{align*}
This shows that all $D_{\alpha}$ for which $\nu_K(\phi_n(\alpha)) > \lambda_n$ intersect on the special
fiber at the point $y=0$, but none of them intersect any $D_{\beta}$ for which $\nu_K(\phi_n(\beta)) =
\lambda_n$.  All such $D_{\beta}$ intersect the reduced special fiber $\aff^1_k \cong \Spec k[y]$ of $\Spec \mc{O}$ at some point where $y \neq 0$.

Now let $\alpha \notin D(\phi_n, \lambda_n)$. We will show that
$D_{\alpha} \cap (\Spec \mc{O})_s$ is empty 
by contradiction. Suppose not. Let $P \in D_{\alpha} \cap (\Spec
\mc{O})_s$ be a closed point of $\Spec \O$.  We have a well-defined
element $g(P) \in k$ for every $g \in \mc{O}$ coming from evaluating
$g$ at $P$. Since $P$ is the closed point of $D_{\alpha} \cong \Spec A$ with $A \subseteq \mc{O}_{K(\alpha)}$, it follows that $g(\alpha) \in \O_{K(\alpha)}$ and furthermore, $g(P) = g(\alpha) \mod \mathfrak{m}_{\O_{K(\alpha)}}$. We will now construct a $g \in \mc{O}$ with $\nu_K(g(\alpha)) < 0$, which is a contradiction. Let $b$ be such that $bv(\phi_n) \in \ints_{> 0}$, and let $g \colonequals
\phi_n^b/\pi_K^{bv(\phi_n)}$.  Then $v(g) = 0$ so $g \in \mc{O}$, but
 \[ \nu_K(g(\alpha)) = b(\nu_K(\phi_n(\alpha)) - v(\phi_n)) < 0.  \]
 
 Thus $D_{\alpha}$ does not intersect the special fiber of $\Spec \mc{O}$,
so $D_{\alpha}$ specializes to a point of $\mc{Y}_s \setminus (\Spec\O)_s$, which is the ``point at infinity'' where $y=\infty$ on the reduced special fiber of $\mc{Y}$.  This finishes the proof.
\end{proof}

\begin{corollary}\label{Cdontmeet}
Let $v = [v_0,\, v_1(\phi_1) = \lambda_1, \ldots,\, v_n(\phi_n) =
\lambda_n]$ be a Mac Lane valuation and let $\mc{Y}$ be a normal model
of $\proj^1_K$ including $v$.   If $\alpha,\beta \in \ol{K}$ are such that $\nu_K(\phi_n(\beta))
\leq \lambda_n \leq \nu_K(\phi_n(\alpha))$ and $\nu_K(\phi_n(\beta))
\neq \nu_K(\phi_n(\alpha))$, then $D_{\alpha}$ and
$D_{\beta}$ do not meet on the special fiber of $\mc{Y}$.
\end{corollary}

\begin{proof}
Immediate from Proposition \ref{Psmallerdiskoid}.
\end{proof}

\begin{corollary}\label{Cannulus}
Let $v = [v_0,\, v_1(\phi_1) = \lambda_1, \ldots,\, v_n(\phi_n) =
\lambda_n]$ and $v' = [v_0,\, v_1(\phi_1) = \lambda_1, \ldots,\, v_n'(\phi_n) =
\lambda_n']$ be Mac Lane valuations with $\lambda_n' < \lambda_n$. Let $\mc{Y}$ be 
a model of $\proj^1_K$ including $v$ and $v'$ on which the $v$- and
$v'$-components intersect, say at a point $z$.  Then $D_{\alpha}$
meets $z$ if and only if $\lambda_n' < \nu_K(\phi_n(\alpha)) < \lambda_n$.
\end{corollary}

\begin{proof}
We may assume $\mc{Y}$ is the $\{v, v'\}$-model $\proj^1_K$.
  Let $\ol{Y}$ and $\ol{Y}'$ be the $v$ and $v'$-components of
$\mc{Y}$, respectively, so that $z = \ol{Y} \cap \ol{Y}'$.  First suppose
$\lambda_n' < \nu_K(\phi_n(\alpha)) < \lambda_n$. If $D_{\alpha}$
meets a point of $\ol{Y} \setminus \ol{Y}'$, then by
Proposition~\ref{Psmallerdiskoid} applied to the blow down of
$\ol{Y}' \subseteq \mc{Y}$ (i.e., the $v$-model of $\proj^1_K$), all $D_{\alpha}$ outside of $D(\phi_n,
\lambda_n)$ intersect this point on $\ol{Y} \subseteq \mc{Y}$.  So
if we blow down $\ol{Y} \subseteq \mc{Y}$, then all $D_{\alpha}$ for $\alpha
\notin D(\phi_n, \lambda_n)$ specialize to the same point. Since we can find $\alpha_1,\alpha_2 \in \ol{K} \setminus D(\phi_n,\lambda_n)$ with $\nu_K(\phi_n(\alpha_1)) = \lambda_n'$ and $\lambda_n' < \nu_K(\phi_n(\alpha_2)) < \lambda_n$, the previous line contradicts Proposition~\ref{Psmallerdiskoid} applied to the $v'$-model of $\proj^1_K$.  The same argument applied to the blow down of $\ol{Y}$
(i.e, the $v$-model of $\proj^1_K$) yields a contradiction if
$D_{\alpha}$ intersects a point of $\ol{Y}' \setminus \ol{Y}$.  So
$D_{\alpha}$ meets the intersection point $z$ of the two irreducible
components of the special fiber.

Now, suppose $\nu_K(\phi_n(\alpha)) \leq \lambda_n'$.  Fix $\beta \in \ol{K}$ such that $\lambda_n' < \nu_K(\phi_n(\beta)) < \lambda_n$. Corollary~\ref{Cdontmeet} shows that $D_{\alpha}$ and $D_{\beta}$ do not meet on the
$v'$-model of $\proj^1_K$, and thus not on $\mc{Y}$ either. In
particular, since $D_{\beta}$ meets $z$ by the previous paragraph,
$D_{\alpha}$ does not. A similar proof works if $\nu_K(\phi_n(\alpha))
\geq \lambda_n$ using the $v$-model instead of the $v'$-model.  This
completes the proof of the corollary.
\end{proof}

\subsection{Resolution of singularities on normal models of
  $\proj^1$}\label{Sresolution} 
 
Let $\mc{Y}$ be a normal model of $\proj^1_K$.  A \emph{minimal
  regular resolution of $\mc{Y}$} is a (proper) regular model $\mc{Z}$ of $\proj^1_K$
with a surjective, birational morphism $\pi: \mc{Z} \to \mc{Y}$ such
that the special fiber of $\mc{Z}$ contains no $-1$-components
(\cite[Definition 2.2.1]{CES}).  Such minimal regular
resolutions exist and are unique, e.g., by \cite[Theorem 2.2.2]{CES}.

In the remainder of \S\ref{Sresolution}, we recall a fundamental result from
\cite{ObusWewers} (which requires $k$ algebraically closed), expressing minimal regular resolutions of models of
$\proj^1_K$ with irreducible special fiber in terms of Mac Lane
valuations. 

\subsubsection{Shortest $N$-paths}
We start by recalling the notion of \emph{shortest $N$-path}, introduced in \cite{ObusWewers}. 

\begin{defn}\label{DNpath}
Let $N$ be a natural number, and let $a > a' \geq 0$ be rational
numbers.  An \emph{$N$-path} from $a$ to $a'$  is a
decreasing sequence $a = b_0/c_0 > b_1/c_1 > \cdots > b_r/c_r = a'$ of rational numbers in lowest terms such that
$$\frac{b_i}{c_i} - \frac{b_{i+1}}{c_{i+1}} = \frac{N}{\lcm(N, c_i)\lcm(N, c_{i+1})}$$ for
$0 \leq i \leq r-1$.  If, in addition, no proper subsequence of $b_0/c_0 > \cdots > b_r/c_r$ containing
  $b_0/c_0$ and $b_r/c_r$ is an $N$-path, then the sequence is called
  the \emph{shortest $N$-path} from $a$ to $a'$.  
  \end{defn}

\begin{remark}
By \cite[Proposition A.14]{ObusWewers}, the shortest $N$-path from
$a'$ to $a$ exists and is unique.
\end{remark}

\begin{remark}\label{R1pathconsec}
  Observe that two successive entries $b_i/c_i > b_{i+1}/c_{i+1}$ of a shortest $1$-path satisfy $b_i/c_i - b_{i+1}/c_{i+1} = 1/(c_ic_{i+1})$.
\end{remark}

\begin{example}\label{Efarey}
The sequence $1 > 1/2 > 2/5 > 3/8 > 1/3 > 0$ is a concatenation
of the shortest $1$-path from $1$ to $3/8$ with the shortest $1$-path
from $3/8$ to $0$.  Note that the denominators increase until $3/8$
and then decrease afterwards.
%As Proposition
%\ref{Pnonincreasing}(iii) guarantees, $3/8$ is the
%mediant of $2/5$ and $1/3$, and $2/5$ is the mediant of $1/2$ and
%$1/3$.  The denominators $8$, $5$, and $2$ form an arithmetic progression
%with increment $-3$.
\end{example}

\subsubsection{Regular resolutions}
The following proposition expresses minimal regular resolutions in
terms of Mac Lane valuations and shortest $N$-paths.  We fix the
following notation.

\begin{notation}\label{Nyfprime}
If $v$ is a Mac Lane valuation, then $\mc{Y}_v$ is the $v$-model of
$\proj^1_K$, and $\mc{Y}_v^{\reg}$ is its minimal
regular resolution.  Furthermore, $\mc{Y}_{v, 0}$ is the $\{v_0,
v\}$-model of $\proj^1_K$, and $\mc{Y}_{v,0}^{\reg} \to \mc{Y}_{v,0}$ is its minimal
regular resolution.  Observe that if $\mc{X}$ is the $v_0$-model of $\proj^1_K$, then contracting the $v$-component of $\mc{Y}_{v,0}$ yields a canonical map
$\mc{Y}_{v,0}^{\reg} \to \mc{X}$ factoring through $\mc{Y}_{v,0}$. 
\end{notation}

\begin{prop}[{\cite[Theorem 7.8]{ObusWewers}}]\label{Pgeneralresolution}
Let $v =
[v_0,\, v_1(\phi_1) = \lambda_1, \ldots,\, v_n(\phi_n) = \lambda_n]$.
For each $i$, write $\lambda_i = b_i/c_i$ in lowest terms, and let
$N_i = \lcm_{j < i} c_j = \deg(\phi_i)$ (Corollary \ref{CvalueofN}).
Set $\lambda_0 = \lfloor \lambda_1 \rfloor$, as well as $N_0 = N_1 = 1$ and $e(v_0/v_{-1}) = 1$.
Then the minimal regular
resolution $\mc{Y}_v^{\reg}$ of $\mc{Y}_v$ is the normal model of $\proj^1_K$ that
includes exactly the following set of valuations:

\begin{itemize}
\item For each $1 \leq i \leq n$, the valuations
$$v_{i,\lambda} \colonequals [v_0,\, v_1(\phi_1) = \lambda_1, \ldots,\,
v_{i-1}(\phi_{i-1}) = \lambda_{i-1},\, v_{i}(\phi_i) =
\lambda],$$ as $\lambda$ ranges through the shortest $N_i$-path
from $\beta_i$ to $\lambda_i$, where $\beta_i$ is 
the least rational number greater than or equal to $\lambda_i$ in
$(1/N_i)\ints = \Gamma_{v_{i-1}}$.  In other words, $\beta = \lceil
N_i \lambda_i \rceil/N_i$.

\item For each $0 \leq i \leq n-1$, the valuations
$$w_{i, \lambda} \colonequals [v_0,\, v_1(\phi_1) = \lambda_1, \ldots,\, v_i(\phi_i) = \lambda_i,\ v_{i+1}(\phi_{i+1}) = \lambda],$$
as $\lambda$ ranges through the shortest $N_{i+1}$-path from
$\lambda_{i+1}$ to $e(v_i/v_{i-1})\lambda_i$, excluding the endpoints.
\item The valuation $\tilde{v}_0 \colonequals [v_0,\, v_1(\phi_1) = \lambda_0]$ (which is just
  $v_0$ if $\lambda_1 < 1$).
\end{itemize}
\end{prop}

\begin{remark}\label{Rvi}
For $\lambda = e(v_i/v_{i-1})\lambda_i$, one sees that $w_{i, \lambda}
= v_i$.
\end{remark}

\begin{remark}\label{Radjacencyres}
For $v$ as in Proposition~\ref{Pgeneralresolution}, consider the set $S$
of valuations included in the minimal regular resolution
$\mc{Y}_v^{\reg}$ of the
$v$-model $\mc{Y}_v$ of $\proj^1_K$.  Using the partial order
$\prec$ on $S$, one constructs a tree whose vertices are the
elements of $S$ and where there is an
edge between two vertices $w$ and $w'$
if and only if $w \prec w'$ and there is no $w''$
with $w \prec w'' \prec w'$.  One can show that this tree is 
the dual graph of $\mc{Y}_v^{\reg}$ by a repeated application of Proposition~\ref{Psmallerdiskoid}.  This graph is shown in Figure~\ref{Fresolutiongraph}.
%For each $v$, by picking $\alpha_v$ specializing to a not $0$ not $\infty$ part of $v$ and using Corollary~\ref{Cannulus} and using reverse inclusion of diskoids, one can compute how the compo%nents of the model described above meet each other, and it looks like the following...
\end{remark}

\begin{figure}
\begin{center}

\setlength{\unitlength}{1.2mm}
\begin{picture}(140,50)

\put(10,45){\circle*{2}}
\put(8,48){$\tilde{v}_0$}
\put(20,45){\circle*{2}}
\put(17,48){$w_{0,\lambda}$}
\put(40,45){\circle*{2}}
\put(37,48){$w_{0,\lambda}$}
\put(50,45){\circle*{2}}
\put(48,48){$v_1$}
\put(60,45){\circle*{2}}
\put(57,48){$w_{1,\lambda}$}
\put(80,45){\circle*{2}}
\put(74,48){$w_{n-2,\lambda}$}
\put(90,45){\circle*{2}}
\put(86,48){$v_{n-1}$}
\put(130,45){\circle{2}}
\put(128,48){$v_n$}
\put(100,45){\circle*{2}}
\put(95,48){$w_{n-1,\lambda}$}
\put(120,45){\circle*{2}}
\put(115,48){$w_{n-1,\lambda}$}
\put(119,41){$v'$}

\put(50,35){\circle*{2}}
\put(43,35){$v_{1,\lambda}$}
\put(50,15){\circle*{2}}
\put(43,15){$v_{1,\lambda}$}
\put(50,5){\circle*{2}}
\put(42,5){$v_{1,\beta_1}$}
\put(90,35){\circle*{2}}
\put(79,35){$v_{n-1,\lambda}$}
\put(90,15){\circle*{2}}
\put(79,15){$v_{n-1,\lambda}$}
\put(90,5){\circle*{2}}
\put(76,5){$v_{n-1,\beta_{n-1}}$}
\put(130,35){\circle*{2}}
\put(123,35){$v_{n,\lambda}$}
\put(131.5,34.5){$v''$}
%\put(131.5,24.5){$v^*$}
\put(130,15){\circle*{2}}
\put(123,15){$v_{n,\lambda}$}
\put(130,5){\circle*{2}}
\put(121,5){$v_{n,\beta_{n}}$}

\put(11,45){\line(1,0){8}}
\put(41,45){\line(1,0){8}}
\put(51,45){\line(1,0){8}}
\put(81,45){\line(1,0){8}}
\put(91,45){\line(1,0){8}}
\put(121,45){\line(1,0){8}}
\put(50,45){\line(0,-1){10}}
\put(90,45){\line(0,-1){10}}
\put(50,15){\line(0,-1){10}}
\put(90,15){\line(0,-1){10}}
\put(130,15){\line(0,-1){10}}
\put(130,44){\line(0,-1){9}}

\linethickness{1pt}
\dottedline{3}(20,45)(40,45)
\dottedline{3}(60,45)(80,45)
\dottedline{3}(100,45)(118,45)
\dottedline{3}(50,35)(50,15)
\dottedline{3}(90,35)(90,15)
\dottedline{3}(130,35)(130,15)

\end{picture}
\end{center}
\caption{The dual graph of the minimal resolution of the $v =
  v_n$-model of $\proj^1_K$.  The white vertex corresponds to the
  strict transform of the $v$.  The vertex labeled $v'$ (resp.\ $v''$)
  corresponds to the successor (resp.\ precursor) valuation of $v_n$, see
  \S\ref{Scontractions}.}\label{Fresolutiongraph}
\end{figure}
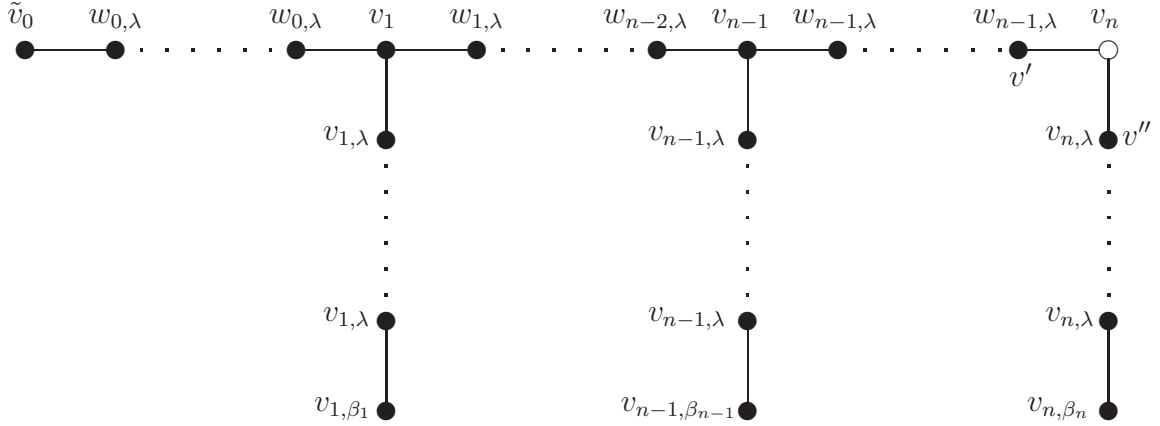

\begin{corollary}\label{Cthrowin0}
With the notation of Proposition~\ref{Pgeneralresolution}, the valuations included in $\mc{Y}_{v,0}^{\reg}$ are the valuations
      included in $\mc{Y}_v^{\reg}$ as well as $v_0$ and the valuations $[v_0,\, v_1(\phi_1) =
      \lambda]$ for $\lambda \in \{1, 2,\ldots, \lambda_0 - 1\}$.
      Equivalently, the valuations included in $\mc{Y}_{v,0}^{\reg}$ are exactly
      the valuations we would get from Proposition~\ref{Pgeneralresolution} if we changed our convention from $\lambda_0 = \lfloor \lambda_1 \rfloor$ to $\lambda_0 = 0$. 
\end{corollary}

\begin{proof}
      If $\lambda_0 = 0$, then $\mc{Y}_v^{\reg}$ includes $v_0$, so $\mc{Y}_v^{\reg} = \mc{Y}_{v,0}^{\reg}$.  If $\lambda_0 \geq 1$, then if $\mc{Z}$ is the normal model of
$\proj^1_K$ including the valuations included in $\mc{Y}$ as well as
$v_0$, then there may be a singularity where the components
corresponding to $v_0$ and $[v_0,\, v_1(\phi_1) = \lambda_0]$ cross.
Since $v_0$ and $[v_0,\, v_1(\phi_1) = 0]$ are the same
valuation, and since $\lambda_0 > \lambda_0 - 1 > \cdots > 1 > 0$ is
the shortest $1$-path from $\lambda_0$ to $0$, \cite[Corollary~7.5]{ObusWewers} shows that resolving this singularity yields exactly
the description of $\mc{Y}_{v,0}^{\reg}$ in the statement of the corollary.  The equivalent description is clear, since $\lambda_1 < 1$
is equivalent to $\lambda_0 = 0$.
\end{proof}

\begin{prop}\label{Pmultiplicitydecreases}
  Let $v = [v_0,\, v_1(\phi_1) = \lambda_1,\, \ldots,\, v_n(\phi_n)
  = \lambda_n]$ be a Mac Lane valuation.  Let $\mc{Y}_{v,0}^{\reg}$ be the minimal regular resolution of the
  $\{v, v_0\}$-model of $\proj^1_K$.  If $w$ is a valuation included in
  $\mc{Y}_{v,0}^{\reg}$, then $e(w/v_0) \leq e(v/v_0)$, and furthermore,
  if $e(w/v_0) = e(v/v_0)$, then $w \preceq v$.
\end{prop}

\begin{proof}
  %\andrew{This proof has been completely rewritten}
For a contradiction, choose a valuation $w$ such that $e(w/v_0)$ is
maximal among those $w$ violating the
proposition, and among these choose $w$ such that $w$ is maximal under
$\preceq$.  

First observe that, since $e(v_i/v_0) \leq e(v_n/v_0)$ for all $i
 \leq n$ and $v_i \preceq v_n$, we may assume
 \begin{equation}\label{E:whereisw}
  w \neq v_i \ \ \textup{ for any } i.
 \end{equation}

  %\padma{This in particular means that $w \notin \{v,v_0\}$ and by Figure~\ref{Fresolutiongraph}, that there are at most two $w'$ adjacent to $w$ in the dual graph for $\mc{Y}_{v,0}^{\reg}$.}  
 
 Let $c_w$ be the
 self-intersection number of the $w$-component of $\mc{Y}_{v,0}^{\reg}$.  Since $\mc{Y}_{v,0}^{\reg}$ is the minimal regular resolution of the $\mc{Y}_{v,v_0}$-model, and since $w \notin \{v,v_0\}$ by (\ref{E:whereisw}), we
 have $c_w \neq -1$, thus $c_w \leq -2$.  By standard intersection
 theory on regular arithmetic surfaces (e.g.,
 \cite[(3.4)]{ObusWewers}), we have 
$$-c_w e(w/v_0) = \sum_{w'} e(w'/v_0),$$ where the sum is taken over
all $w'$ such that the $w'$-component intersects the $w$-component.
Since $w \neq v_i$ for any $i$ by (\ref{E:whereisw}), Figure~\ref{Fresolutiongraph} shows that there are at most two such $w'$.

Since $-c_w \geq 2$ and by assumption $e(w/v_0) \geq e(w'/v_0)$ for all $w'$ in the
sum, we find that there are exactly two $w'$ and
$e(w'/v_0) = e(w/v_0)$ for each of them.  By
Remark~\ref{Radjacencyres}, one of the $w'$ satisfies $w \prec w'$.
Since $w$ is maximal under $\prec$, we conclude that $w'$ does not
violate the proposition.  But $w \prec w'$ and $e(w/v_0) = e(w'/v_0)$ imply that $w$ does
not violate the proposition either, a contradiction. \qedhere

\end{proof}

\subsection{Valuations related to a given Mac Lane valuation}\label{Scontractions} 
Let $v = [v_0,\, v_1(\phi_1) = \lambda_1,\ldots,\, v_n(\phi_n) = \lambda_n]$.
Recall that $c_i$ is the denominator of $\lambda_i$, when written in lowest terms.  Let $N_n = \lcm(c_1, \ldots, c_{n-1}) = \deg(\phi_n)$ (Corollary
\ref{CvalueofN}).  We assume that $n \geq 1$ and $\lambda_n \notin
\Gamma_{v_{n-1}}  = (1/N_n)\ints$.

Let $\mc{Y}_v$ be the $v$-model of $\proj^1_K$, and let
$\mc{Y}_v^{\reg}$ be its minimal regular resolution.  By Proposition
\ref{Pgeneralresolution}, the following Mac Lane valuations are
included in $\mc{Y}_v^{\reg}$:
\begin{itemize}
\item $v' \colonequals w_{n-1,\lambda'} =  [v_0,\, v_1(\phi_1) = \lambda_1,\ldots,\,
  v_{n-1}(\phi_{n-1}) = \lambda_{n-1},\, v_n'(\phi_n) = \lambda'],$
\item $v'' \colonequals v_{n,\lambda''} =  [v_0,\, v_1(\phi_1) = \lambda_1,\ldots,\,
  v_{n-1}(\phi_{n-1}) = \lambda_{n-1},\, v_n''(\phi_n) = \lambda''],$
\end{itemize}
where $\lambda'$ is the entry directly following
$\lambda_n$ in the shortest
$N_n$-path from $\lambda_n$ to $e(v_{n-1}/v_{n-2})\lambda_{n-1}$, and
$\lambda''$ is the entry directly preceeding $\lambda_n$ in the
shortest $N_n$-path from $\lceil N_n \lambda_n \rceil/N_n$ to $\lambda_n$.  The
valuation $v'$ (resp.\ $v''$) is called the \emph{successor} (resp.\
\emph{precursor}) valuation to $v$.
%If $\lambda' \notin \Gamma_{v_{n-1}}  = (1/N_n)\ints$, we introduce
%the \emph{precursor} valuation $v^*$ of $v'$, defined by
%$$v^* \colonequals v_{n,\lambda^*} = [v_0,\, v_1(\phi_1) = \lambda_1,\ldots,\,
%    v_{n-1}(\phi_{n-1}) = \lambda_{n-1},\, v_n^*(\phi_n) =
%    \lambda^*],$$
%    where $\lambda^*$ is the entry directly preceeding $\lambda'$ in
%    the shortest $N_n$-path from $\lceil N_n\lambda' \rceil/N_n$ to $\lambda'$.  Note that
%    $\lambda^* > \lambda_n$, so $v^* \neq v$, see Remark \ref{Rfullinequality}.

In the description of the \textit{minimal} embedded resolution of an irreducibile horizontal divisor, we have to analyze various contractions of the model described in Proposition~\ref{Pgeneralresolution} (Definition~\ref{D3types}, Definition~\ref{Dstar}) and the specializations of horizontal divisors on them (Lemma~\ref{Lspecializations}). For this purpose, it is helpful to introduce some convenient notation to refer to certain special valuations appearing in Proposition~\ref{Pgeneralresolution}. Let us write
\begin{itemize}
\item $v^{*} \colonequals w_{n-1,\lambda^{*}} =  [v_0,\, v_1(\phi_1) = \lambda_1,\ldots,\,
  v_{n-1}(\phi_{n-1}) = \lambda_{n-1},\, v_n'(\phi_n) = \lambda^*],$
\item $v^{**} \colonequals v_{n,\lambda^{**}} =  [v_0,\, v_1(\phi_1) = \lambda_1,\ldots,\,
  v_{n-1}(\phi_{n-1}) = \lambda_{n-1},\, v_n''(\phi_n) = \lambda^{**}],$
\end{itemize}
where $v^*$ (resp.\ $v^{**}$) can represent any of the valuations amongst 
the $w_{n-1, \lambda}$ (resp.\ the $v_{n, \lambda}$) from Proposition~\ref{Pgeneralresolution}. Later in Definition~\ref{Dstar}, we will specialize to specific choices of $v^*$ and $v^{**}$ depending on which regular model is being considered. In the remainder of this section, we establish some inequalities bounding $\lambda^*,\lambda^{**}$ that are valid for all choices of $v^*, v^{**}$. These will then be used in the proof of Lemma~\ref{Lindividualterms}, which is the key technical input for the main theorem of the paper.

Since $\lambda_n \notin (1/N_n)\ints$, we have
$\lfloor N_n \lambda_n \rfloor \leq N_n\lambda' < N_n \lambda_n < N_n
\lambda'' \leq \lceil N_n \lambda_n \rceil$, the first inequality
coming from \cite[Corollaries A.7, A.11]{ObusWewers}.
%Similarly, we have
%$N_n \lambda' < N_n \lambda^* \leq \lceil N_n \lambda_n \rceil$ when
%$\lambda' \notin \Gamma_{v_{n-1}} = (1/N_n)\ints$.
Write $\tilde{\lambda}_n$ (resp.\ $\tilde{\lambda}'$,
$\tilde{\lambda}''$, $\tilde{\lambda}^*$, $\tilde{\lambda}^{**}$)
for $N_n\lambda_n - \lfloor N_n \lambda_n \rfloor$ (resp.\
$N_n\lambda' - \lfloor N_n \lambda_n \rfloor$, $N_n\lambda'' - \lfloor
N_n \lambda_n \rfloor$, $N_n\lambda^* - \lfloor
N_n \lambda_n \rfloor$, $N_n\lambda^{**} - \lfloor
N_n \lambda_n \rfloor$).  Then we obtain 
\[
0 \leq \tilde{\lambda}' < \tilde{\lambda}_n  <
\tilde{\lambda}'' \leq 1.
\]
%Similarly, when $\lambda' \notin \Gamma_{v_{n-1}} = (1/N_n)\ints$, we have
%\[
%0 \leq \tilde{\lambda}' < \tilde{\lambda}^* \leq 1.
%\]

\begin{prop}\label{Pprimespaths}
%For any statement involving $\tilde{\lambda}^*$ below, assume
%$\lambda' \notin \Gamma_{v_{n-1}} = (1/N_n)\ints$.
Let $e$, $e'$, $e''$, $e^*$, and $e^{**}$ be the denominators of
          $\tilde{\lambda}_n$, $\tilde{\lambda}'$,
          $\tilde{\lambda}''$, $\tilde{\lambda}^*$, and $\tilde{\lambda}^{**}$, respectively.
  \begin{enumerate}[\upshape (i)]
\item  The number $\tilde{\lambda}'$ immediately follows $\tilde{\lambda}_n$
  in the shortest $1$-path from $\tilde{\lambda}_n$ to $0$.
\item The number $\tilde{\lambda}''$ immediately preceeds
  $\tilde{\lambda}_n$ in the shortest $1$-path from $1$ to
  $\tilde{\lambda}_n$.
\item The number $\tilde{\lambda}^*$ is on the shortest $1$-path from
  $\tilde{\lambda}_n$ to $0$.
\item The number $\tilde{\lambda}^{**}$ is on the shortest $1$-path from
  $1$ to $\tilde{\lambda}_n$.
  
%      \item We have $\tilde{\lambda}^* > \tilde{\lambda}_n$.
%        \item We have $e' \mid (e - e^*)$.
  \end{enumerate}
\end{prop}

\begin{proof}
  By \cite[Lemma A.7]{ObusWewers}, $N_n \lambda'$ immediately follows
$N_n\lambda_n$ in the shortest $1$-path from $N_n\lambda_n$ to
$N_ne(v_i/v_{i-1})\lambda_{n-1}$, and thus in the shortest $1$-path
from $N_n\lambda_n$ to $\lfloor N_n \lambda_n \rfloor$ by \cite[Lemma A.11]{ObusWewers}.   
Since translating by an integer preserves shortest $1$-paths,
subtracting $\lfloor N_n \lambda_n \rfloor$ from all entries of these
paths yields part (i).  Part (ii) follows similarly, using that
$N_n\lambda''$ immediately precedes $N_n\lambda_n$ in the shortest
$N_n$-path from $\lceil N_n \lambda_n \rceil$ to $N_n \lambda_n$.
The proofs of parts (iii) and (iv) are essentially the same as the
proofs of parts (i) and (ii), respectively.
%Taking the shortest $1$-path from $1$ to $\tilde{\lambda}_n$
%and then appending $\tilde{\lambda}'$ yields a $1$-path $P$ from $1$ to
%$\tilde{\lambda}'$ by part (i).
%By part (ii), $\tilde{\lambda}^*$ lies somewhere in this path, so it
%suffices to show $\tilde{\lambda}^* \neq \tilde{\lambda}_n$.
%In fact, by Proposition \ref{Pnonincreasing}(i) applied to $\tilde{\lambda}'$
%with $i=0$, we see that $e^* < e'$.
%On the other hand, by part (i) and Proposition \ref{Pnonincreasing}(ii) applied to
%$e$, we have $e' < e$.  So $e > e^*$, implying
%$\tilde{\lambda}^* \neq \tilde{\lambda}_n$, proving part (iii).
%
%Part (iv) follows from Proposition \ref{Pnonincreasing}(iii) applied
%to the segment of $P$ between $\tilde{\lambda}^*$ and
%$\tilde{\lambda}'$, where $c_1' = e'$, $c_0 = e$, and $c_i = e^*$ in
%that proposition.
\end{proof}

\begin{example}\label{Eroles}
 If $\tilde{\lambda}_n = 3/8$, we would have
  $\tilde{\lambda}' = 1/3$ and
  $\tilde{\lambda}'' = 2/5$ (cf.\ Example \ref{Efarey}).  We could
  take $\tilde{\lambda}^*$ to be $1/3$ or $0$, and we could take
  $\tilde{\lambda}^{**}$ to be $2/5$, $1/2$, or $1$.
\end{example}

\begin{corollary}\label{Cindexrelations}
Let $e$, $e'$, $e''$, $e^*$, and $e^{**}$ be as in Proposition
\ref{Pprimespaths}.  Then 
\begin{enumerate}[\upshape (i)]
\item $\lambda_n - \lambda' = 1/(N_nee')$.
\item $\lambda'' - \lambda_n = 1/(N_nee'')$.
\item $\lambda_n - \lambda^* \geq 1/(N_nee^*)$, with equality if and
  only if $\lambda^* = \lambda'$.
\item $\lambda^{**} - \lambda_n \geq 1/(N_nee^{**})$, with equality
  if and only if $\lambda^{**} = \lambda''$.
\end{enumerate}
\end{corollary}

\begin{proof}
By Proposition~\ref{Pprimespaths}(i) and the definition of $1$-path, $\tilde{\lambda}_n - \tilde{\lambda}' =
1/(ee')$, from which part (i) follows.  Part (ii) follows similarly,
using Proposition~\ref{Pprimespaths}(ii).  To prove part (iii), note that
Proposition~\ref{Pprimespaths}(iii) shows that $\tilde{\lambda}^*$ is on
the shortest $1$-path from $\tilde{\lambda}_n$ to $0$, but that $\tilde{\lambda}^*$
does not directly follow $\tilde{\lambda}_n$ on this path unless
$\lambda^* = \lambda'$.  The definition of
shortest $1$-paths shows that $\tilde{\lambda}_n -
\tilde{\lambda}^* = 1/ee^*$ if and only if
$\lambda^* = \lambda'$.  Since $\tilde{\lambda}_n - \tilde{\lambda}^*$
is a multiple of $1/ee^*$ by common denominators, part (iii) follows. The proof
of part (iv) is exactly the same, using $\lambda^{**}$, $\lambda''$,
and Proposition~\ref{Pprimespaths}(iv) instead of $\lambda^*$,
$\lambda'$, and Proposition~\ref{Pprimespaths}(iii).\qedhere

%By construction, the numbers $e$, $e'$, and $e''$ are the denominators of
%$N_n\lambda_n$, $N_n \lambda'$, and $N_n\lambda''$ respectively.
% By Remark \ref{Rfullinequality} and the definition of
%$1$-path, we have
%$N_n\lambda_n - N_n \lambda' = 1/(ee')$ and $N_n \lambda'' - N_n \lambda' =
%1/(ee'')$, which proves the corollary.
\end{proof}

\begin{lemma}\label{Lpreindexrelations}
Let $v$, $v'$, $v''$, $v^{*}$, and $v^{**}$ be as above.
If $e$, $e'$, $e''$, $e^*$, and $e^{**}$ are defined as in Proposition
\ref{Pprimespaths}, then $e = e(v/v_{n-1})$, $e' = e(v'/v_{n-1})$,
$e'' = e(v''/v_{n-1})$, $e^* = e(v^*/v_{n-1})$, and $e^{**} = e(v^{**}/v_{n-1})$.
\end{lemma}

\begin{proof}
By construction, $e$ is the
denominator of $N_n\lambda_n$ (and similarly for $e'$, 
$e''$, $e^{*}$, and $e^{**}$).  By \cite[Lemma 5.3(ii)]{ObusWewers}, $e(v/v_0) = \lcm(N_n, c_n)$, where $c_n$ is the
denominator of $\lambda_n$.  By
\cite[Lemma A.6]{ObusWewers}, this is equal to $N_ne$.  Since $N_n =
e(v_{n-1}/v_0)$, we have $e = e(v/v_{n-1})$.  This proves the lemma for $e$,
and the proofs for $e'$, $e''$, $e^*$, and $e^{**}$ are identical.  
\end{proof}

\section{Some regular models of $\P^1$ attached to a polynomial}\label{Shorizontal}
Throughout \S\ref{Shorizontal}, we assume that the residue field $k$ of $K$ is \emph{algebraically closed}.

Let $\alpha \in \mc{O}_{\ol{K}}$ such that $\nu_K(\alpha) > 0$ and the
minimal polynomial $f(x) \in K[x]$ of $\alpha$ has degree at least
$2$. In this section, we first define a canonical Mac Lane valuation $v_f$ attached to $f$. We then define certain natural contractions of the minimal regular resolution of the $v_f$-model, 
called ``Type I'', ``Type II'', or ``Type III'' models. 
These are candidate models for the horizontal divisor $D_{\alpha}$ to
be regular on. We prove some technical results about these three kinds of models. These results will then be used in the next section to show that the minimal regular model on which $D_{\alpha}$ is regular is a special kind of Type I or Type II model. 

% The main goal of this section is
%to construct a regular model
%of $\proj^1_K$ on which the horizontal divisor $D_{\alpha}$
%(\S\ref{Shorizontaldivs}) is regular.  This model is called $\mc{Y}_{v_f'}^{\reg}$,
%and is defined in Notation~\ref{Nyfprime}.  The regularity statement about the horizontal
%divisor is Theorem \ref{Thorizontalregular}.  

\subsection{The Mac Lane valuation associated to a polynomial}\label{Sregularmodelprelims}
Write
$$v_f = [v_0,\, v_1(\phi_1) = \lambda_1,\, \ldots,\, v_n(\phi_n) =
  \lambda_n] $$
  for the unique Mac Lane valuation on $K(x)$ over which $f$ is a proper
key polynomial (Proposition \ref{Pbestlowerapprox}(iv)).  As usual, write $v_0, v_1,\ldots, v_n = v_f$ for the intermediate
valuations.  For $1 \leq i \leq n$,
write $\lambda_i = b_i/c_i$ in lowest terms.  Let $N_i = \lcm(c_1,
\ldots, c_{i-1}) = \deg(\phi_i)$ (Corollary \ref{CvalueofN}).
Furthermore, pick once and for all a root $\alpha$ of $f$.

\begin{remark}\label{Rrelatephitotrunc} If the roots of $f$ generate a tame extension, it is easy to read off the polynomials $\phi_i$ and integers $\lambda_i$ from the truncations of Newton-Puiseux expansions of the roots of $f$ with respect to some choice of uniformizer $t$, as we now explain. Using Proposition~\ref{Pbestlowerapprox}(iii), we see that we can take $\phi_i$ to be the minimal polynomials of the truncations of the Newton-Puiseux expansions just before there is a jump in the lcm of the denominators of the exponents in the expansion. If $\alpha$ is a root of $f$, then Corollary~\ref{Cpseudoeval} shows that $\lambda_i = \nu_K(\phi_i(\alpha)) = \sum_{\phi_i(\beta)=0} \nu_K(\alpha-\beta)$. If $\deg(\phi_i)=m$, then the Galois group of the splitting field of the tame extension generated by the roots of $\phi_i$ is generated by the automorphism $t^{1/m} \mapsto \zeta_m t^{1/m}$ for a primitive $m^{\mathrm{th}}$ root of unity $\zeta_m$. Since the induced $\Z/m\Z$-action on the roots of $\phi_i$ is transitive, a direct computation then shows that for each root $\beta$ of $\phi_i$, the quantity $\nu_K(\alpha-\beta)$ is equal to one of the the exponents in $\alpha$ where the lcm of the denominators of the exponents jumps. (This is the content of \cite[Lemma~8.13]{PadmaTame} using the language of characteristic/jump exponents.)

For example, let $K=\mathbb{C}((t))$ and let $f$ be the minimal
polynomial of $2t-t^{5/2}+t^{8/3}-3t^{7/2}+t^{23/6}$. Then $v_f$ has
the form $$v_f = [v_0,\, v_1(\phi_1) = \lambda_1,v_2(\phi_2) =
\lambda_2] ,$$ and we can take $\phi_1=x-2t$ and $\phi_2$ to be the
minimal polynomial of $2t-t^{5/2}$, with $\lambda_1=5/2$ and $\lambda_2=5/2+8/3$. This example also shows that $\deg(\phi_i)$ and the invariants $\lambda_i$ contain the same information as the characteristic exponents of the Newton-Puiseux expansion of a root of $f$ as in \cite[Example~8.13]{PadmaTame} in the tame case.
\end{remark}

For the rest of this section we will use the following notation.  
\begin{notation}\label{Nmaclane}
Lemma~\ref{Lvnvprime} implies that we are in the situation of
\S\ref{Scontractions}. Like in  \S\ref{Scontractions}, let
\begin{itemize}
\item $v_f' = [v_0,\, v_1(\phi_1) = \lambda_1,\, \ldots,\,
  v_{n-1}(\phi_{n-1}) = \lambda_{n-1},\, v_n'(\phi_n) = \lambda']$
\item   $v_f'' = [v_0,\, v_1(\phi_1) = \lambda_1,\, \ldots,\,
  v_{n-1}(\phi_{n-1}) = \lambda_{n-1},\, v_n'(\phi_n) = \lambda'']$
\end{itemize}
be the successor and precursor valuations to $v_f$, respectively.  
\end{notation}

For simplicity, we write $e = e(v_f/v_{n-1})$, $e' = e(v_f'/v_{n-1})$,
and $e'' = e(v_f''/v_{n-1})$. This is consistent with the notation in
Lemma \ref{Lpreindexrelations} and Proposition
\ref{Pprimespaths}. We record for later usage that $e =
\deg(f)/\deg(\phi_n)$ by Lemma \ref{Lfdegree}(ii). With this notation, we are ready to state the main result of this paper. We postpone the proof to Section~\ref{Sproof}.

\begin{theorem}\label{Thorizontalregular}
Let $f \in \mc{O}_K[x]$ be a monic irreducible polynomial of degree
$\geq 2$, and let
$\mc{X}$ be the $v_0$-model of $\proj^1_K$.  Let $v_f$ be the unique
Mac Lane valuation over which $f$ is a key polynomial, and let $v_f'$
and $v_f''$ be the valuations defined in Notation~\ref{Nmaclane}.  For any Mac
Lane valuation $v$, let $\mc{Y}_{v, 0}^{\reg}$ be defined as in Notation~\ref{Nyfprime}.

\begin{enumerate}[\upshape (i)]
\item If $e(v_f'/v_0) \leq e(v_f''/v_0)$, then the minimal embedded
  resolution of $(\mc{X}, \divi_0(f))$ is $c \colon \mc{Y}_{v_f', 0}^{\reg} \to \mc{X}$, where $c$ is the canonical contraction from Notation~\ref{Nyfprime}.
\item If $e(v_f'/v_0) > e(v_f''/v_0)$, then the minimal embedded
  resolution of $(\mc{X}, \divi_0(f))$ is $c \colon \mc{Y}_{v_f'', 0}^{\reg}
  \to \mc{X}$, where $c$ is the canonical contraction from Notation~\ref{Nyfprime}.
\end{enumerate}
\end{theorem}

We now give two basic examples illustrating Theorem~\ref{Thorizontalregular}.

\begin{example}\label{E38res}
  If $f = x^8 - \pi_K^3$, then $v_f = [v_0, \, v_1(x) = 3/8]$.  As in Example~\ref{Efarey}, the shortest $1$-path from $1$ to $3/8$ is given by $1 > 1/2 > 3/8$ and the shortest $1$-path from $3/8$ to $0$ is given by $3/8 > 1/3 > 0$, yielding $v_f' = [v_0,\, v_1(x) = 1/3]$ and $v_f'' = [v_0,\, v_1(x) = 1/2]$.  Since $e(v_f'/v_0) = 3$ and $e(v_f''/v_0) = 2$, part (ii) of Theorem~\ref{Thorizontalregular} applies, and the minimal embedded resolution of $(\mc{X}, \divi_0(f))$ is $\mc{Y}_{v_f'', 0}^{\reg} \to \mc{X}$.
\end{example}

\begin{remark}
In Example~\ref{E38res}, applying Proposition~\ref{Pgeneralresolution} shows that $\mc{Y}_{v_f'', 0}^{\reg}$ includes the valuations $v_{\lambda} := [v_0,\, v_1(x) = \lambda]$, for $\lambda \in \{0, 1/2, 1\}$.  In particular, there exist $\lambda$ both greater than and less than $3/8$ for which $v_{\lambda}$ is included.  By Corollary~\ref{Cannulus}, this implies that $D_{\alpha}$, for $\alpha$ a root of $f$, specializes to the intersection of two components (the ones corresponding to $\lambda = 0$ and $\lambda = 1/2$).  This property makes $\mc{Y}_{v_f'', 0}$ a prototype for what we will call a ``Type I model'' in the sequel.  In particular, $v_{1/2}$ is one of the $v_{1, \lambda}$ and $v_{0}$ is one of the $w_{0, \lambda}$; see Definition~\ref{D3types}(i).
\end{remark}

\begin{example}\label{Estandardregular}
If $f$ is Eisenstein, then $v_f = [v_0,\, v_1(x) = 1/\deg(f)]$.  The shortest $1$-path from $1/\deg(f)$ to $0$ is given by $1/\deg(f) > 0$ and the shortest $1$-path from $1$ to $1/\deg(f)$ is given by $1 > 1/2 > \cdots > 1/(\deg(f) - 1) > 1/\deg(f)$.  So $v_f' = [v_0,\, v_1(x) = 0] = v_0$, and $v_f'' = [v_0,\, v_1(x) = 1/(\deg(f) - 1)]$.  Since $e(v_f'/v_0) = 1$ and $e(v_f''/v_0) = \deg(f) - 1$, part (i) of Theorem~\ref{Thorizontalregular} applies, and the minimal embedded resolution of $(\mc{X}, \divi_0(f))$ is $\mc{Y}_{v_f', 0}^{\reg} \to \mc{X}$.  But $\mc{Y}_{v_f', 0}^{\reg} = \mc{Y}_{v_0, 0}^{\reg} = \mc{X}$, so this recovers the easy-to-verify fact that if $\alpha$ is a root of $f$, then $D_{\alpha}$ is regular on the $v_0$-model $\mc{X}$ of $\proj^1_K$.
\end{example}  

\begin{remark}
In Example~\ref{Estandardregular}, since $\mc{Y}_{v_f', 0}^{\reg} = \mc{X}$ has irreducible special fiber, $D_{\alpha}$, for $\alpha$ a root of $f$, specializes to only one irreducible component.   This property makes $\mc{Y}_{v_f'', 0}$ a prototype for what we will call a ``Type II model'' in the sequel.  In particular, $v_0$ is one of the $w_{0, \lambda}$; see Definition~\ref{D3types}(ii).
\end{remark}

\subsection{The model $\mc{Y}_{v_f}^{\reg}$ and its contractions}\label{Sgoodmodel}

As we start contracting components in $\mc{Y}_{v_f}^{\reg}$ to identify the \textit{minimal} embedded resolution of the the pair $(\P^1_{\mc{O}_K}, \divi_0(f))$, we go through an intermediate sequence of regular models of $\P^1_K$ that naturally breaks up into three types (Definition~\ref{D3types}), based on the specialization behaviour of $D_{\alpha}$ (Proposition~\ref{Palphaspecialize}). To understand whether $D_{\alpha}$ is regular on these contractions, we also need to understand where some closely related divisors specialize on each of these three types of models (Corollary~\ref{Cnophii}). The goal of the rest of the subsection is to prove Proposition~\ref{Pfhorizontal}, which lets us write down an explicit function that cuts out the divisor $D_{\alpha}$ on each of these three types of models -- the forms of the explicit functions look different in each of these three cases, hence the subdivision. In \S\ref{Sproof} we will finally use these explicit functions to understand the regularity of $D_{\alpha}$ on each of these three types of models. We will show that the minimal embedded resolution has to be one of the Type I or Type II models, and $D_{\alpha}$ is not regular on the unique Type III model.

%We now define three types of regular contractions of $\mc{Y}_{v_f}^{\reg}$.  
We use the notation of Proposition~\ref{Pgeneralresolution} and Figure~\ref{Fresolutiongraph}.

\begin{defn}\label{D3types}
  Fix $f$ as in this section.  The Type I, II, and III models below
  implicitly depend on $f$. 
  \begin{itemize}
  \item  A \emph{Type I} model of $\proj^1_K$ is any regular contraction of $\mc{Y}_{v_f}^{\reg}$
    that includes at least one of the $v_{n, \lambda}$ and one of the
    $w_{n-1, \lambda}$, but does not
    include $v_f$.
  \item A \emph{Type II} model of $\proj^1_K$ is any regular contraction of
    $\mc{Y}_{v_f}^{\reg}$ that does not include $v_f$ or any $v_{n,
      \lambda}$, but does include at least one of the $w_{n-1,
      \lambda}$.
  \item Assuming that $\mc{Y}_{v_f}^{\reg}$ includes at least one valuation other than $v_f$, the $v_{n, \lambda}$, and the $w_{n-1, \lambda}$, we define \textbf{the} \emph{Type III} model of $\proj^1_K$ to be the model where the
    $v_f$-component is contracted, as well as all the $v_{n, \lambda}$
    and the $w_{n-1, \lambda}$.
  \end{itemize}
\end{defn}

\begin{remark}
  Since $v_{n-1}$ is one of the $w_{n-1, \lambda}$, one sees that the Type III model is the contraction of the $v_{n-1}$-component in
  $\mc{Y}_{v_{n-1}}^{\reg}$.  
\end{remark}

\begin{defn}\label{Dstar} \hfill
\begin{itemize}
\item  Given a Type I or Type II model $\mc{Y}$, define $$v_f^* =
  [v_{n-1}, v_f^*(\phi_n) = \lambda^*] = w_{n-1,
  \lambda^*},$$ where $\lambda^*$ is maximal such that $w_{n-1,
  \lambda^*}$ is included in
$\mc{Y}$.  
\item  Given a Type I model $\mc{Y}$, define $$v_f^{**} =
  [v_{n-1}, v_f^{**}(\phi_n) = \lambda^{**}] = v_{n,
  \lambda^{**}},$$ where $\lambda^{**}$ is minimal such that $v_{n,
  \lambda^{**}}$ is included in
$\mc{Y}$.
\item If $v_f^*$ (resp.\ $v_f^{**}$) is defined, define $e^*$ (resp.\
  $e^{**}$) to be the denominator of $N_n\lambda^*$ (resp.\
  $N_n\lambda^{**}$). Note that this notation is consistent with that
  of Proposition~\ref{Pprimespaths}.  
\item Given a Type III model $\mc{Y}$, define $v_{n-1}'$ and
  $v_{n-1}''$ to be the successor and precursor valuations to
  $v_{n-1}$, respectively.
\end{itemize}
\end{defn}

\begin{remark}
 Note that the $v_f^*$ and
$v_f^{**}$-components of $\proj^1_K$ intersect using
Proposition~\ref{Pgeneralresolution} and Remark~\ref{Radjacencyres}.
\end{remark}

\begin{remark}
By Lemma~\ref{Lpreindexrelations}, $e^* = e(v_f^*/v_{n-1})$ and $e^{**}
= e(v_f^{**}/v_{n-1})$.
\end{remark}

\subsubsection{Specializations of horizontal divisors}
\begin{lemma}\label{Luniquespec}
On the model $\mc{Y}_{v_f}^{\reg}$ the only component of the special fiber
that $D_{\alpha}$ meets is the $v_f$-component.
\end{lemma}

\begin{proof}
The multiplicity of the $v_f$-component of $\mc{Y}_{v_f}^{\reg}$
in the special fiber is $e(v_n/v_{n-1})e(v_{n-1}/v_0)$.  But $e(v_n/v_{n-1}) = \deg(f)/\deg(\phi_n)$ by
Lemma \ref{Lfdegree}(ii) and $e(v_{n-1}/v_0) = \deg(\phi_n)$ by Corollary~\ref{CvalueofN}.  So the
multiplicity is equal to $\deg(f)$.

By Proposition~\ref{Cpseudoeval},
$v_f(\phi_n(\alpha)) = \lambda_n$.  So by
\cite[Lemma~7.3(iii)]{ObusWewers} and
Proposition~\ref{Psmallerdiskoid}, $D_{\alpha}$ intersects a regular
point $z$ on
the $v_f$-model of $\proj^1_K$, which is also a smooth point of the
reduced special fiber by \cite[Lemma~7.1]{ObusWewers}. By the previous
line, we conclude that the point $z$ is not part of the base locus of
the projection $\mc{Y}_{v_f}^{\reg} \to \mc{Y}_{v_f}$, and this proves the lemma. 
\end{proof}

\begin{lemma}\label{Lspecializations}
  Let $y$ be a point on the $v_f$-component of $\mc{Y}_{v_f}^{\reg}$.
  \begin{enumerate}[\upshape (i)]
  \item Suppose $\mc{Y}$ is a Type I model, and $\tau \colon
    \mc{Y}_{v_f}^{\reg} \to \mc{Y}$ is the standard contraction.  Then
    $\tau(y)$ lies on the intersection of the $v_f^*$- and $v_f^{**}$-components of $\mc{Y}$. 
  \item Suppose $\mc{Y}$ is a Type II model, and $\tau \colon
    \mc{Y}_{v_f}^{\reg} \to \mc{Y}$ is the standard contraction.  Then
    $\tau(y)$ lies only on the $v_f^*$-component of $\mc{Y}$. 
  \item  Suppose $\mc{Y}$ is the Type III model, and $\tau \colon
    \mc{Y}_{v_f}^{\reg} \to \mc{Y}$ is the standard contraction.  Then
    $\tau(y)$ lies on the intersection of the $v_{n-1}'$- and
    $v_{n-1}''$-components of $\mc{Y}$.   
  \end{enumerate}

\end{lemma}

\begin{proof}
This follows from Remark~\ref{Radjacencyres} and Figure~\ref{Fresolutiongraph}.
\end{proof}

\begin{prop}\label{Palphaspecialize}
Let $\alpha$, $f$, $v_f$, $v_f^*$, $v_f^{**}$, $v_{n-1}'$, and $v_{n-1}''$ be as in this section.
  \begin{enumerate}[\upshape (i)]
     \item If $\mc{Y}$ is a Type I model of $\proj^1_K$, then the divisor
        $D_{\alpha}$ on $\mc{Y}$ meets the intersection of the two components of
        the special fiber corresponding to $v_f^*$ and $v_f^{**}$.
      \item If $\mc{Y}$ is a Type II model of $\proj^1_K$, then the
        divisor $D_{\alpha}$ on $\mc{Y}$ intersects only the $v_f^*$-component of the special fiber. 
     \item If $\mc{Y}$ is the Type III model of $\proj^1_K$, then the
       divisor $D_{\alpha}$ on $\mc{Y}$ meets the intersection of the
       two components of the special fiber corresponding to $v_{n-1}'$ and $v_{n-1}''$.
      \end{enumerate}
\end{prop}

\begin{proof}
By Lemma~\ref{Luniquespec}, $D_{\alpha}$ meets the special fiber of
$\mc{Y}_{v_f}^{\reg}$ only on the $v_f$-component.  Parts (i), (ii),
and (iii) of the proposition now follow from the respective parts of
Lemma~\ref{Lspecializations}.
% By Corollary~\ref{Cpseudoeval} applied to $v_f$, we have
%$\nu_K(\phi_n(\alpha)) = \lambda_n$.  By definition, $\lambda^* <
%\lambda_n < \lambda^{**}$, and
%Corollary~\ref{Cannulus} applied to
%$v_f^*$ and $v_f^{**}$ now shows that $D_{\alpha}$ meets the intersection
%point of the two desired components, proving part (i).
%
%For part (ii), note that Lemma~\ref{Luniquespec} shows that
%$D_{\alpha}$ specializes only to the $v_f$-component on
%$\mc{Y}_{v_f}^{\reg}$.  Let $\tau \colon \mc{Y}_{v_f}^{\reg} \to
%\mc{Y}$ be the standard contraction.  
%Looking at the diagram accompanying
%Proposition~\ref{Pgeneralresolution} \andrew{or the other diagram that
%  will be inserted?}, we see that
%$\tau(z)$ meets only the $v_f^*$-component of
%$\mc{Y}$, which proves part (ii).
%
%For part (iii), note that $\mc{Y}$ is a contraction of any Type I or
%Type II model.  By part (i) or part (ii), all components that
%$D_{\alpha}$ specializes to in a Type I or Type II model are
%contracted in the Type III model.  From Figure \ref{Fmain} \andrew{or the
%  figure to be inserted}, we see that $D_{\alpha}$
%meets the intersection of the $v_{n-1}'$- and $v_{n-1}''$-components of
%the special fiber of $\mc{Y}$.  
\end{proof}

\begin{corollary}\label{Cnophii}
Let $\mc{Y}$ be a Type I or Type II model of $\proj^1_K$.  Let
$\alpha_n$ be a root of $\phi_n$.
  \begin{enumerate}[\upshape (i)]
  \item Suppose $\beta \in \ol{K}$ has degree less than $\deg(\phi_n)$
    over $K$.  Then $D_{\alpha}$ and
$D_{\beta}$ do not meet on the special fiber of $\mc{Y}$.
\item If $\mc{Y}$ is Type I, then $D_{\alpha}$ and
  $D_{\alpha_n}$ do not meet on the special fiber of $\mc{Y}$.
\item If $\mc{Y}$ is Type II or Type III, then $D_{\alpha}$ and $D_{\alpha_n}$
  meet on the special fiber of $\mc{Y}$.
\end{enumerate}
\end{corollary}

\begin{proof}
By Proposition~\ref{Palphaspecialize}, $D_{\alpha}$ specializes to the
$v_f^*$-component of the special fiber of $\mc{Y}$.  By
Corollary~\ref{CvalueofN} and Lemma~\ref{Lpreindexrelations}, the multiplicity
of this component is $N_n e^* = \deg(\phi_n)e^* \geq \deg(\phi_n)$.  So by
\cite[Lemma~5.1(a)]{LL}, $D_{\beta}$ does not specialize to this
component.  This proves part (i).

To prove part (ii), assume $\mc{Y}$ is Type I.  Note that $\alpha_n$ is a root of $\phi_n$, we have
$\nu_K(\phi_n(\alpha_n)) = \infty$, which does not lie between $\lambda^*$ and
$\lambda^{**}$.  As a consequence,
Corollary \ref{Cannulus} and Proposition~\ref{Palphaspecialize}(i)
show that $D_{\alpha}$ does not meet $D_{\alpha_n}$ on the special
fiber of $\mc{Y}$.

To prove part (iii), it suffices to assume $\mc{Y}$ is Type II, since
a Type III model is a contraction of a Type II model.  Since
both $\nu_K(\phi_n(\alpha)) = \lambda_n$ and $\nu_K(\phi_n(\alpha_n))
= \infty$ are greater than
$\lambda^*$, Proposition~\ref{Psmallerdiskoid} shows that they meet on
the special fiber of the $v_f^*$-model of $\proj^1_K$.  This point
is not a base point of the contraction $\mc{Y} \to \mc{Y}_{v_f^*}$,
because that would violate Proposition~\ref{Palphaspecialize}(ii).
Thus, $D_{\alpha}$ and $D_{\alpha_n}$ meet on $\mc{Y}$.  
\end{proof}

\subsubsection{}
The final result of this section, Proposition~\ref{Pfhorizontal}, shows how to appropriately modify the function $f$ to make a function that precisely cuts out the divisor $D_{\alpha}$ on each of the three types of models -- 

\begin{prop}\label{Pfhorizontal}
Let $\mc{Y}$ be a Type I, Type II, or Type III model of $\proj^1_K$, and let
$v_f^*$ and $v_f^{**}$ be defined accordingly. 
\begin{enumerate}[\upshape (i)]
\item If $\mc{Y}$ is Type I, the quantity $b := e(\lambda_n -
  \lambda^*)/(\lambda^{**} - \lambda^*)$ is an integer.  Furthermore,
there exists a monomial $s$ in $\phi_1, \ldots, \phi_{n-1}$ over $K$ such
  that the divisor $D_{\alpha}$ is locally cut out by $sf/\phi_n^b$.
\item If $\mc{Y}$ is Type II, there exists a monomial $t$ in
  $\phi_1,\ldots, \phi_{n-1}$ such the divisor $D_{\alpha}$ is
  locally cut out by $sf$, where $s = t^e$.
\item If $\mc{Y}$ is Type III, then there exists
  $s \in K(x)$ such that the divisor $D_{\alpha}$ is locally cut out by $sf$, and such
  that the support of $s$ is locally (near $D_{\alpha}$) contained in
  the special fiber of $\mc{Y}$. 
\end{enumerate}
%  
%
%
%\item We have $r \in \nats$, and there is a monomial $t$ in $\phi_1,\phi_2,\ldots,\phi_{n-1}$ over $K$
%such that the horizontal divisor $D_{\alpha}$ obtained by taking the
%closure of the point $\alpha$ in $\mc{Y}_{v_f'}^{\reg}$ 
%is locally cut out by the divisor of $t^r f/\phi_n^b$.
%\item Let $s = t^r$.  Then $v_f'(sf/\phi_n^b) = 0$.  Furthermore, if $\lambda' \notin \Gamma_{v_{n-1}}$, then
%  $v_f^*(sf/\phi_n^b) = 0$ as well.
%\item We have $v_f'(t \phi_n^{e'}) = 0$.
%\end{enumerate}
\end{prop}

\begin{remark}\label{Rallthree}
Since $\phi_1, \ldots, \phi_{n-1}$ all have degree lower than
$\deg(\phi_n)$, Corollary~\ref{Cnophii}(i) shows that the support of $s$ is locally (near $D_{\alpha})$ contained in the
special fiber of $\mc{Y}$ in parts (i) and (ii), as well as part (iii).
\end{remark}

To prove Proposition~\ref{Pfhorizontal}, we first need to compute the orders of vanishing of various auxiliary functions that will be used to modify the function $f$ along vertical components of Type I, II, and III models.  This is accomplished in Lemma~\ref{Lfphie}. The proof also needs two other short lemmas (Lemma~\ref{Pestar} and Lemma~\ref{Pvaluegroup}).

\begin{lemma}\label{Lfphie}
  Let $f = \phi_n^e + a_{e-1}\phi_n^{e-1} + \cdots + a_0$ be the
  $\phi_n$-adic expansion of $f$.  Let $\mc{Y}$ be a Type I or Type II
  model of $\proj^1_K$, and let $v_f^*$ and $v_f^{**}$ be defined
  accordingly.  Let $a_e = 1$.
  \begin{enumerate}[\upshape (i)]
  \item We have $v_f^*(f) = v_f^*(\phi_n^e) = e \lambda^*$.
  \item We have $v_f^*(a_i\phi_n^i) > e \lambda^*$ for $0 \leq i \leq e-1$.
  \item In the case of a Type I model, we have $v_f^{**}(f) = v_f^{**}(a_0) = e \lambda_n$.
  \item In the case of a Type I model, we have $v_f^{**}(a_i \phi_n^i)
    > e\lambda_n$ for $1 \leq i \leq e$.    
%  \item In the case of a Type II model, if $\beta$ is a root of $f - \phi_n^e$, then $D_{\beta}$ does
%    not meet $D_{\alpha}$ on the special fiber of
%    $\mc{Y}$.
  \end{enumerate}  
\end{lemma}

\begin{proof}
%By Lemma~\ref{Lfdegreebasic}, we have $v_f(f) = v_f(\phi_n^e) = v_f(a_0) =
%e\lambda_n$, and $v_f(a_i\phi_n^i) \geq e\lambda_n$ for $1 \leq i \leq
%e-1$.  It remains to prove the first equality in part (i).  Now, $v_f^*(a_i\phi_n^i) = v_f(a_i \phi_n^i)  - i(\lambda_n -
%\lambda^*) \geq e\lambda_n - i(\lambda_n - \lambda^*)$ for $1 \leq i
%\leq e-1$.  \andrew{Resume editing here.} By
%Corollary~\ref{Cindexrelations}(i), this equals $e\lambda_n -
%i/(N_ne'e)$, which is strictly greater than $e\lambda_n - 1/(N_ne')$.  Since
%$1/(N_ne')$ generates $\Gamma_{v_f'}$, and $e\lambda_n = v_f(a_0) =
%v_f'(a_0) \in \Gamma_{v_f'}$, we in fact have that $v_f'(a_i\phi_n^i)
%\geq e\lambda_n$ for $1 \leq i \leq e-1$.  This proves part (i).
%
By Lemma~\ref{Lfdegreebasic}, $\phi_n^e$ is a term in the $\phi_n$-adic
expansion of $f$ with minimal $v_f$-valuation. It is also the term
whose valuation is decreased the most when $v_f$ is replaced with
$v_f^*$.  Thus $\phi_n^e$ is the unique term in the $\phi_n$-adic
expansion of $f$ with minimal $v_f^*$-valuation.  Since $v_f^*(\phi_n)
= \lambda^*$ by definition, this proves parts (i) and (ii).

Similarly, by Lemma~\ref{Lfdegreebasic}, $a_0$ is a term in the
$\phi_n$-adic expansion of $f$ with minimal $v_f$-valuation.  It is
also the term whose valuation is increased the least when $v_f$ is
replaced by $v_f^{**}$.  Thus $a_0$ is the unique term in the
$\phi_n$-adic expansion of $f$ with minimal $v_f^{**}$-valuation.
Since $v_f^{**}(a_0) = v_f(a_0) = e\lambda_n$ (Lemma~\ref{Lfdegreebasic}),
this proves parts (iii) and (iv).
%For part (iii), one has that $v_f^*(a_i \phi_n^i) > v_f(a_i \phi_n^i)
%\geq v_f(a_0) = v_f^*(a_0)$
%for $1 \leq i \leq e$, where we take $a_n =1$.  So $v_f^*(f) =
%v_f^*(a_0)$, and $v_f^*(f - \phi_n^e) = v_f^*(a_0)$ as well.  Since
%$v_f^*(a_0) = v_f(a_0) = v_f(f) = e\lambda_n$, by
%Lemma~\ref{Lfdegreebasic}, this proves (iii).
%
%By part (i), Corollary~\ref{Cnoroot} applies to to $f- \phi_n^e$ and
%$v_f'$.  Thus $\beta$, being a root of $f - \phi_n^e$, has
%$\nu_K(\phi_n(\beta)) \leq \lambda'$.  If $\lambda_n \notin \Gamma_{v_{n-1}}$, then
%Proposition~\ref{Palphaspecialize}(i) shows that $D_{\alpha}$ meets the special
%fiber of $\mc{Y}_{v_f'}^{\reg}$ at the intersection of the $v_f'$ and $v_f^*$-components.  By Corollary~\ref{Cannulus}, the
%same is not true of $D_{\beta}$, proving part (iv) in this case.  If,
%$\lambda_n \in \Gamma_{v_{n-1}}$, then
%Proposition~\ref{Palphaspecialize}(ii) shows that $D_{\alpha}$ meets
%the special fiber of $\mc{Y}_{v_f'}^{\reg}$ only on the $v_f'$-component.  Since $\nu_K(\phi_n(\alpha)) = \lambda_n > \lambda'$,
%(Corollary~\ref{Cpseudoeval}) while $\nu_K(\phi_n(\beta)) \leq \lambda'$
%as we have seen, Corollary~\ref{Cdontmeet} shows that $D_{\alpha}$ and
%$D_{\beta}$ do not meet on the $v_f'$-model of $\proj^1_K$.  Thus they
%do not meet on $\mc{Y}_{v_f'}^{\reg}$, finishing part (iv).
\end{proof}

\begin{lemma}\label{Pestar}
On a Type I model $\mc{Y}$, we have $\lambda^{**} - \lambda^* = 1/(N_ne^*e^{**})$. 
\end{lemma}

\begin{proof}
Since $\mc{Y}$ is regular and the $v_f^*$- and $v_f^{**}$-components
intersect, \cite[Corollary~7.6]{ObusWewers} (with $\mc{X} = \mc{X}'$
there) shows that $\lambda^{**} > \lambda^*$ is the shortest
$N_n$-path.  By \cite[Corollary~A.7]{ObusWewers},
$\tilde{\lambda}^{**} > \tilde{\lambda}^*$
is a shortest $1$-path, where $\tilde{\lambda}^*$ and
$\tilde{\lambda}^{**}$ are as in Proposition~\ref{Pprimespaths}.  By
the definition of a $1$-path, $\tilde{\lambda}^{**} -
\tilde{\lambda}^{*} = 1/(e^*e^{**})$, so $\lambda^{**} - \lambda^{*} = 1/(N_ne^*e^{**})$. 
\end{proof}

\begin{lemma}\label{Pvaluegroup}
On a Type II model $\mc{Y}$, we have $\Gamma_{v_f^*} = \Gamma_{v_{n-1}}$.
\end{lemma}

\begin{proof}
If $\mc{Y}$ is a Type II model, then it dominates
$\mc{Y}_{v_f^*}^{\reg}$, and thus includes all the valuations therein.  On the other hand, by the definition of a Type II model,
$\mc{Y}$ does not include any valuation of the form $[v_0,\,
v_1(\phi_1) = \lambda_1,\, \ldots,\, v_{n-1}(\phi_{n-1}) =
\lambda_{n-1},\, v_n(\phi_n) = \lambda]$ with $\lambda > \lambda^*$.
Applying Proposition~\ref{Pgeneralresolution} to $v_f^*$, this forces the $\beta_n$ referred to in the first bullet point of
Proposition~\ref{Pgeneralresolution} to equal $\lambda^*$.  So
$\lambda^* \in (1/N_n)\ints = \Gamma_{v_{n-1}}$. Since $\Gamma_{v_f^*} = [v_{n-1}, v_i(\phi_i) =\lambda^*]$, it follows that $\lambda^*$ together with $\Gamma_{v_{n-1}}$ generates $\Gamma_{v_f^*}$. Combining the previous two sentences, we get
$\Gamma_{v_f^*} = \Gamma_{v_{n-1}}$. 
\end{proof}

\begin{proof}[Proof of Proposition~\ref{Pfhorizontal}]
To prove the first assertion of part (i), note that $\lambda^{**} -
\lambda^* = 1/(N_ne^*e^{**})$ by Lemma~\ref{Pestar}.  So $b = N_nee^*e^{**}(\lambda_n -
\lambda^*)$.  Since the denominator of $\lambda_n$ divides $e(v_f/v_0)
= N_ne$ and that
of $\lambda^*$ divides $e(v_f^*/v_0) = N_ne^{*}$, we have that $b$
is an integer, and is in fact divisible by $e^{**}$.

Now, assume we are on a Type I model and let $y$ be the point where $D_{\alpha}$ meets the special fiber of $\mc{Y}$, i.e., the specialization of $f(x) = 0$.  The
function $f$ in general does not locally cut out $D_{\alpha}$ at $y$, because
$\divi(f)$ might also include vertical components passing through
$y$. By Proposition~\ref{Palphaspecialize}(i), $z$ is the intersection of
the $v_f^*$ and $v_f^{**}$-components of the special fiber.  By
Corollary~\ref{Cnophii}(i), (ii), the specialization of $\phi_i = 0$ is not $y$ for any
$1 \leq i \leq n$.  So to finish the proof of part (i), it suffices to construct $s$
as in the proposition such that $v_f^*(sf/\phi_n^b) =
v_f^{**}(sf/\phi_n^b) = 0$.

By Lemma~\ref{Lfphie}(i), we have $v_f^*(f/\phi_n^b) = (e -
b)\lambda^*$.  Likewise, by Lemma~\ref{Lfphie}(iii), we have
$v_f^{**}(f/\phi_n^b) = e\lambda_n - b\lambda^{**}$.
Since $e^{**} \mid b$, and the denominators of $\lambda_n$ and
$\lambda^{**}$ are $N_ne$ and $N_ne^{**}$ respectively, $e\lambda_n -
b\lambda^{**} \in \Gamma_{v_{n-1}} = (1/N_n)\ints$.
This means that there exists $s$ as in the proposition such that
$v_f^{**}(sf/\phi_n^b) = 0$.  Since $v_f^{*}(s) = v_f^{**}(s)$,
showing that $v_f^*(sf/\phi_n^b) = 0$ is reduced to showing that $(e - b)\lambda^* = e\lambda_n -
b\lambda^{**}$.  But this is immediate upon plugging in the definition
of $b$.

Now we prove part (ii).  Let $y$ be as in part (i).  By
Proposition~\ref{Palphaspecialize}(ii), $y$ lies on a unique component of the special fiber, namely the $v_f^*$-component. 
Furthermore, since the value group of $v_f^*$ is $\Gamma_{v_{n-1}}$
(Lemma~\ref{Pvaluegroup}), we
have that $v_f^*(\phi_n) = \lambda^* \in \Gamma_{v_{n-1}}$. Thus we can find $t$ as in
the proposition such that $v_f^*(t) = -\lambda^*$.     
By Lemma~\ref{Lfphie}(i), $v_f^*(t^ef) = v_f^*(sf) = 0$.  By Corollary~\ref{Cnophii}(i), the specialization of $\phi_i = 0$ is not
$y$ for any $1 \leq i \leq n-1$.  So $sf$ cuts out $D_{\alpha}$,
proving part (ii).

For part (iii), note that $\mc{Y}$ is regular,
and is thus a local UFD.  Since $\divi(f)$ and $D_{\alpha}$ agree on the
generic fiber in a neighborhood of $D_{\alpha}$, there exists $s \in
K(x)$ with the desired property.
\end{proof}

\section{Minimal embedded resolution}\label{Sproof}

In this section, we prove our main result,
Theorem~\ref{Thorizontalregular}, which explicitly gives the minimal
embedded resolution of $(\mc{X}, \divi_0(f))$, where $\mc{X}$ is the
$v_0$-model of $\proj^1_K$ and $f \in \mc{O}_K[x]$ is a monic
polynomial of degree at least $2$.  We begin in \S\ref{Slocalrings} with some general results
on regularity, and then return to Mac Lane valuations and models of
$\proj^1_K$ for the proof in \S\ref{Scriterion}.
The main technical lemma that makes everything work is
Lemma~\ref{Lindividualterms}, which depends heavily on the work in \S\ref{Shorizontal}.

Throughout \S\ref{Sproof}, with the exception of Remark~\ref{Retaledescent} and the conclusion of the paper immediately following it, we assume that the residue field $k$ of $K$ is \emph{algebraically closed}.

\subsection{Generalities on regular models}\label{Slocalrings}

\begin{lemma}\label{Lcanblowup}
If $\mc{X}$ is a regular model of $\proj^1_K$ and $D$ is a reduced,
effective, regular divisor on $\mc{X}$ and if $f: \mc{X}' \to \mc{X}$ is a modification,
then the strict transform $D'$ of $D$ in $X'$ is regular.
\end{lemma}

\begin{proof}
  Since $\mc{X}$ is normal, $f$ is an isomorphism above points of codimension 1, thus over the generic point of each component of $D$.  So $D' \to D$ is proper and birational.  Since $\dim(D') = \dim(D) = 1$, $D' \to D$ is finite as well, and thus it is an isomorphism, proving the lemma.
%  By \cite[Theorem 9.2.2]{LiuBook}, the map $f$ is a finite
%sequence of blowups at reduced closed points.
%\padma{In general, the factorization theorem for birational maps only say every birational morphism factors as blowups along ideal sheaves (\cite[Theorem
%8.1.24]{LiuBook}), right? So in theory there can be blowups along
%codimension $1$ subvarieties/fat points too. The Mac Lane description
%might be over kill, but from it, we know for maps between normal
%models of $\P^1_K$, blowups along closed points suffice?}
%Since the blow up of a
%regular scheme at a closed point is regular (\cite[Lemma
%8.1.19(a)]{LiuBook}), we reduce to the case where $f$ is a single
%blowup at a closed point.  But then $D'$ is a blowup of $D$ at the
%same point (\cite[Tag
%080E]{StacksProject}), proving the lemma.
\end{proof}
%\padma{or \cite[Corollary~8.1.17]{LiuBook}} \andrew{I think the Stacks Project reference is a bit clearer.}

The following proposition is well-known, but we were unable to find an
exact reference.  We state it only in the generality we need.
\begin{prop}\label{Puniqueminimal}
  If $\mc{X}$ is a regular model of $\proj^1_K$ and $D$ is an integral horizontal
  divisor on $\mc{X}$, then there is a unique minimal modification $\mc{X}' \to \mc{X}$ such that
  $\mc{X}'$ is regular and the strict transform of $D$ is regular.
\end{prop}

\begin{proof}
By \cite[Theorem 9.2.26]{LiuBook}, there exists \emph{some} modification 
$\mc{Y} \to \mc{X}$ with $\mc{Y}$ regular under which the total transform of $D$ has normal crossings, and in particular, the strict transform of $D$ is
thus regular.  We now prove that a minimal such $\mc{Y}$ is unique.
By \cite[Theorem 9.2.2]{LiuBook}, the morphism $\mc{Y} \to \mc{X}$ is
a finite sequence of blowups at reduced closed points.

We now prove the proposition by induction on the minimum number $n$ of
blowups of $\mc{X}$ at closed points required to make the strict transform of $D$
regular.  The case $n = 0$ is trivial.   If not, since blowups in
centers outside Supp$(D)$ do not affect $D$, any minimal sequence of
blowups making the strict transform of $D$ regular begins with blowing
up the (unique) intersection point $x$ of $D$ with the special fiber
of $\mc{X}$.  Replacing $\mc{X}$ with its blowup at $x$, and noting that the strict transform of $D$ is still integral on this blowup and then
applying induction completes the proof.
\end{proof}

\begin{lemma}\label{Ldivstructure}
Let $\mc{Y}$ be a regular snc-model of a smooth curve $Y$ over $K$, and
let $y \in \mc{Y}$ be a closed point.
  Let $\divi(f)$, $\divi(g)$ be the divisors in $\Spec \hat{\mc{O}}_{\mc{Y},y}$ of functions $f, g \in
  \hat{\mc{O}}_{\mc{Y}, y}$ respectively.
  \begin{enumerate}[\upshape (i)]
  \item
    Suppose $\divi(f)$ is of the form $\sum_{i=1}^r c_i D_i$
  for some integers $c_i \geq 0$ where the $D_i$ are Weil prime
  divisors.  If $\sum_i c_i \geq 2$, then 
  $f \in \mf{m}_{\mc{Y}, y}^2$.
\item Suppose $\divi(f) = D$ where $D$ is a Weil prime divisor
  corresponding to one of the irreducible components of the special
  fiber of $\mc{Y}$ passing through $y$.  Then $f \in \mf{m}_{\mc{Y},
    y} \setminus \mf{m}_{\mc{Y}, y}^2$.
\item
  If $\divi(f) = D$ and $\divi(g) = E$, where $D$ and $E$ are Weil
  prime divisors corresponding to two different components of the
  special fiber of $\mc{Y}$ passing through $Y$, then the images of
  $f$ and $g$ are linearly independent in $\mf{m}_{\mc{Y}, y} /
  \mf{m}_{\mc{Y}, y}^2$.
\end{enumerate}
\end{lemma}

\begin{proof}
 First note that the regular local
ring $\hat{\mc{O}}_{\mc{Y}, y}$ is a UFD and thus every height one
prime ideal is principal.  Thus in the situation of part (i), $f = w\prod_i f_i^{c_i}$, where $w$ is a
unit and $\divi(f_i) = D_i$.  Since the $f_i$ lie in the maximal
ideal, this proves (i).

In fact, by \cite[Lemma~2.3.2 and its proof]{CES}, we can write
  $$\hat{\mc{O}}_{\mc{Y}, y} \cong \mc{O}_K[[y_1, y_2]]/(y_1^{m_1}
  \cdots y_r^{m_r} - u\pi_K),$$ with $r \in \{1, 2\}$.  The irreducible components of the
  special fiber passing through $y$ are cut out by $y_1$ if $r=1$ and, by $y_1$ and $y_2$ if $r =
  2$.  So in the situation of part (ii), we have $f = wy_1$
  or $f = wy_2$, with $w$ a unit.  Since $\mf{m}_{\mc{Y}, y} =  (y_1,
  y_2)$, this proves part (ii).  In the situation of
  part (iii), we have $r = 2$, and the result follows from the fact
  that the images of $w_1y_1$ and $w_2y_2$ are linearly independent in
  $(y_1, y_2)/(y_1, y_2)^2$.
\end{proof}

\begin{prop}\label{Pm2}
  Let $\mc{Y}$ be a regular model of $\proj^1_K$, and let $y$ be the
  point where $D_{\alpha}$ intersects the special fiber.  Let $g \in
  \hat{\mc{O}}_{\mc{Y}, y}$ be such that $\divi(g) = D_{\alpha}$ on
  $\Spec \hat{\mc{O}}_{\mc{Y}, y}$.  Then
  $D_{\alpha}$ is regular if and only if $g \notin \mf{m}_{\mc{Y},y}^2$.   
\end{prop}

\begin{proof}
This is \cite[Corollary~4.2.12]{LiuBook}.
% Regularity of $D_{\alpha}$ is equivalent to regularity of
% $\hat{\mc{O}}_{\mc{Y}, y} / (g)$, that is, that the maximal ideal of
% $\hat{\mc{O}}_{\mc{Y}, y}$ is generated by $g$ and some other
% element $h$.  By Nakayama's lemma, this is equivalent to
% $\mf{m}_{\mc{Y},y}/\mf{m}_{\mc{Y},y}^2$ being generated by the image
% $\ol{g}$ of $g$ and some other element $\ol{h}$.  As
% $\mf{m}_{\mc{Y},y}/\mf{m}_{\mc{Y},y}^2$ is a dimension $2$ vector
% space (since $\hat{\mc{O}}_{\mc{Y}, y}$ is regular), this is
% equivalent to $\ol{g} \neq 0$, which proves the proposition.  
\end{proof}

\subsection{Non-archimedean analysis of valuations in an expansion}\label{Scriterion}

%\padma{It is hard to directly check if a function cutting out the divisor $D_{\alpha}$ is in square of the maximal ideal at its unique closed point $y$ since $\mf{m}_{\mc{Y},y}$ is a $2$-dimensional regular local ring. In this section, we instead understand when terms in the natural decomposition of such a function are in $\mf{m}_{\mc{Y},y}^2$ by first computing their orders of vanishing along vertical components through $y$ and then using Lemma~\ref{Ldivstructure}.}

Maintain our notation
from \S\ref{Shorizontal}.  In particular, for the remainder of the paper, $f \in \mc{O}_K[x]$ is monic
and irreducible of degree at least $2$, $\alpha$ is a root of $f$, and on any regular model of
$\proj^1_K$, the divisor $D_{\alpha}$ is the horizontal divisor
corresponding to $\alpha$ as in \S\ref{Shorizontaldivs}.  As in
\S\ref{Shorizontal}, we use the notation $v_f$ for the unique Mac Lane
valuation over which $f$ is a proper key polynomial.  We also use the
valuations $v_f'$ and $v_f''$ from Notation~\ref{Nmaclane}, and we use
the concept of Type I/II/III models associated to $f$ from
Definition~\ref{D3types}, which give rise to valuations $v_f^*$,
$v_f^{**}$, $v_{n-1}'$, and $v_{n-1}''$ as in Definition~\ref{Dstar}.
As in \S\ref{Shorizontaldivs}, we write $e = e(v_f/v_{n-1})$, $e' = e(v_f'/v_{n-1})$, $e'' =
e(v_f''/v_{n-1})$, and, when there is a Type I/II model in play, $e^*
= e(v_f^*/v_{n-1})$ and $e^{**} = e(v_f^{**}/v_{n-1})$.

We decompose the function cutting out the unique horizontal divisor of $D_{\alpha}$ using the $\phi_n$-adic expansion of $f$, and analyze which of the terms in the decomposition are in $\mf{m}_{\mc{Y},y}^2$ for Type I/II models $\mc{Y}$. This will be the key technical input for analyzing regularity of $D_{\alpha}$ on these models in the next section. 

\begin{lemma}\label{Lindividualterms}
Let $\mc{Y}$ be a Type I or Type II model of $\proj^1_K$, and let $y
\in \mc{Y}$ be the point where $D_{\alpha}$ meets the special fiber of $\mc{Y}$.
Let $s$ be as in Proposition~\ref{Pfhorizontal}(i), (ii), let $b$ be as in Proposition~\ref{Pfhorizontal}(i) if $\mc{Y}$ is Type
I and let $b = 0$ if $\mc{Y}$ is Type II.  If $f = \phi_n^e +
a_{e-1}\phi_n^{e-1} + \cdots + a_0$ is the $\phi_n$-adic expansion of
$f$, then we can write
\begin{equation}\label{Eexpansion}
\frac{sf}{\phi_n^b} = s\phi_n^{e-b} + sa_{e-1}\phi_n^{e-1-b} +
\cdots + sa_0\phi_n^{-b}.
\end{equation}
Then,
\begin{enumerate}[\upshape (i)]
\item All terms $s a_i \phi_n^{i - b}$ of (\ref{Eexpansion}) for $1
  \leq i \leq e-1$ are in $\mf{m}_{\mc{Y}, y}^2$. 
\item We have %$sa_0 \phi_n^{-b} \in \mf{m}_{\mc{Y}, y}$, and
  $sa_0\phi_n^{-b} \in \mf{m}_{\mc{Y}, y}^2$ if and only if $v_f^* \neq
  v_f'$.
\item We have %$s\phi_n^{e-b} \in \mf{m}_{\mc{Y}, y}$, and
  $s\phi_n^{e-b} \in \mf{m}_{\mc{Y}, y}^2$ if and only if $\mc{Y}$ is
  Type II or $v_f^{**} \neq v_f''$.
\item Suppose $\mc{Y}$ is Type I.  If $v_f^* = v_f'$ and $v_f^{**} = v_f''$, then $s
  \phi_n^{e-b}$ and $sa_0\phi_n^{-b}$ generate linearly independent elements of
  $\mf{m}_{\mc{Y}, y}/ \mf{m}_{\mc{Y}, y}^2$.
\end{enumerate}
\end{lemma}

\begin{proof}
Let $y$ be the point where $D_{\alpha}$ intersects the special fiber
of $\mc{Y}$.  Recall from Proposition~\ref{Pfhorizontal} that
$sf/\phi_n^b$ cuts out $D_{\alpha}$.  By Remark~\ref{Rallthree}, the horizontal part of
$\divi(s)$ does not contain $y$.  The
same is true for all of the $\divi(a_i)$, since the $a_i$ have degree
less than $\phi_n$ by definition.  Furthermore, Corollary~\ref{Cnophii}(ii) shows
that the same is true for the horizontal part of $\divi(\phi_n)$ if
$\mc{Y}$ is Type I.  

By Proposition~\ref{Palphaspecialize}, $y$ is the intersection of the
$v_f^*$- and $v_f^{**}$-components if $\mc{Y}$ is Type I, and $y$ lies
on only the $v_f^*$-component of $\mc{Y}$ is Type II.  Write $D^*$ and
$D^{**}$ for the prime divisors corresponding to the $v_f^*$- and $v_f^{**}$-components, respectively.

We now prove part (i).  Assume $1 \leq i \leq e-1$. 
%\padma{Since $b=e(\lambda_n-\lambda^*)/(\lambda^{**}-\lambda^*)$ and} 
By
Lemma~\ref{Lfphie}(i), (ii), 
%\padma{$e\lambda^* = v_f^*(\phi_n^e) = $ } 
$v_f^*(f) < v_f^*(a_i\phi_n^i)$,
and since the divisor of $sf/\phi_n^b$ is horizontal by construction, so $0 = v_f^*(sf/\phi_n^b) < v_f^*(sa_i\phi_n^{i-b})$.  Thus $D^*$ lies in
the support of $sa_i\phi_n^{i-b}$.  If $\mc{Y}$ is Type I, the same is
true for $D^{**}$ using Lemma~\ref{Lfphie}(iii), (iv).  Since no
horizontal component of $\divi(sa_i\phi_n^{i-b})$ passes through $y$, we have that $sa_i\phi_n^{i-b} \in \hat{\mc{O}}_{\mc{Y},y}$ and thus,
Lemma~\ref{Ldivstructure}(i) shows that $sa_i\phi_n^{i-b} \in
\mf{m}_{\mc{Y}, y}^2$.  On the other hand, if $\mc{Y}$ is Type II,
then Corollary~\ref{Cnophii}(iii) shows that the horizontal part of
$\divi_0(\phi_n)$ does pass through $y$.  In this case,
Lemma~\ref{Ldivstructure}(ii) shows that $sa_i\phi_n^i \in
\mf{m}_{\mc{Y}, y}^2$.  This concludes the proof of part (i).

For part (ii), Lemma~\ref{Lfphie}(i), (ii) show as above that $D^*$ is in the
support of $\divi(sa_0\phi_n^{-b})$.  If $\mc{Y}$ is Type I, then Lemma~\ref{Lfphie}(iii) shows
that $v_f^{**}(f) = v_f^{**}(a_0)$, so $0 = v_f^{**}(sf/\phi_n^b) =
v_f^{**}(sa_0\phi_n^{-b})$, meaning that $D^{**}$ is \emph{not} in the
support of $\divi(sa_0\phi_n^{-b})$.  Observe further that the horizontal support of $\divi(sa_0\phi_n^{-b})$
does not pass through $y$, regardless of whether $\mc{Y}$ is Type I or
Type II.  This means that we have $sa_i\phi_n^{i-b} \in \hat{\mc{O}}_{\mc{Y},y}$ and by
Lemma~\ref{Ldivstructure}(i), we thus have $sa_0\phi_n^{-b} \in
\mf{m}_{\mc{Y}, y}^2$ if and only if the multiplicity of
$D^*$ in $\divi(sa_0\phi_n^{-b})$ is at least $2$.  

By Corollary~\ref{CvalueofN} and Lemma~\ref{Lpreindexrelations}, the
multiplicity of $D^*$ in the special fiber is $N_ne^*$,
so its multiplicity in $\divi(sa_0\phi_n^{-b})$ is
$N_ne^*v_f^*(sa_0\phi_n^{-b})$.  Since $v_f^*(sf/\phi_n^b) = 0$,
$v_f^*(a_0) = v_f^{**}(a_0) = e\lambda_n$ (Lemma~\ref{Lfphie}(iii)),
and $v_f^*(f) = e \lambda^*$ (Lemma~\ref{Lfphie}(i)), we
have

\begin{align*}
  N_n e^* v_f^*(sa_0 \phi_n^{-b}) &= N_n e^* v_f^*(a_0/f) \\
  &= N_n e^* (e\lambda_n - e\lambda^*) \\
                                  &\geq 1.
\end{align*}
where the inequality follows from Corollary~\ref{Cindexrelations}(iii), and equality holds
if and only if $v_f^* = v_f'$.  So the multiplicity of $D^*$ in
$\divi(sa_o\phi_n^{-b})$ is at least $2$ if and only if $v_f^* \neq v_f'$, finishing
part (ii).

For part (iii), first suppose $\mc{Y}$ is Type I.  Then
Lemma~\ref{Lfphie}(i), (iii), (iv) show using similar reasoning to part
(ii) that $D^{**}$ is in the support of
$\divi(s\phi_n^{e-b})$ but $D^*$ is not.  This proves the first
assertion of part (iii).  Since the horizontal part of
$\divi(s\phi_n^{e-b})$ does not pass through $y$, the same reasoning
as in part (ii) reduces us to showing that the multiplicity of
$D^{**}$ in $\divi(s\phi_n^{e-b})$ is at least 2 if and only if
$v_f^{**} \neq v_f''$.

By Corollary~\ref{CvalueofN} and Lemma~\ref{Lpreindexrelations}, the
multiplicity of $D^{**}$ in the special fiber is $N_ne^{**}$,
so its multiplicity in $\divi(s\phi_n^{e-b})$ is
$N_ne^{**}v_f^{**}(s\phi_n^{e-b})$.  Since $v_f^{**}(sf/\phi_n^b) = 0$
and $v_f^{**}(f) = e \lambda_n$ (Lemma~\ref{Lfphie}(i)), we
have

\begin{align*}
  N_n e^{**} v_f^{**}(s\phi_n^{e-b}) &= N_n e^{**} v_f^{**}(\phi_n^e/f) \\
  &= N_n e^{**} (e\lambda^{**} - e\lambda_n) \\
                                  &\geq 1.
\end{align*}
where the inequality follows from Corollary~\ref{Cindexrelations}(iv), and equality holds
if and only if $v_f^{**} = v_f''$.  So the multiplicity of $D^{**}$ in
$\divi(s\phi_n^{e-b})$ is at least 2 if and only if $v_f^{**} \neq
v_f''$, proving part (iii) in this case.

Now suppose $\mc{Y}$ is Type II.  Then $s \phi_n^{e-b} = s\phi_n^e$,
and by Corollary~\ref{Cnophii}(iii), the horizontal part of
$\divi(s\phi_n^e)$ does meet $y$.  By Proposition~\ref{Pfhorizontal},
$s$ can be taken to be an $e$th power in $K[x]$.  Since $e \geq 2$, we
have $s \phi_n^e \in \mf{m}_{\mc{Y}, y}^e \subseteq \mf{m}_{\mc{Y},
  y}^2$, finishing the proof of part (iii).

Lastly, by the proofs of parts (ii) and (iii), if $v_f^* = v_f'$ and
$v_f^{**} = v_f''$, then $\divi(s \phi_n^{e-b}) =
D^{**}$ and $\divi(sa_0\phi_n^{-b}) = D^*$ in $\Spec
\hat{\mc{O}}_{\mc{Y}, y}$.  Applying
Lemma~\ref{Ldivstructure}(iii) completes the proof of part (iv).
\end{proof}

\begin{lemma}\label{LtypeIII}
Assume the Type III model $\mc{Y}$ of $\proj^1_K$ exists.  Let $s$ be as in
Proposition~\ref{Pfhorizontal}(iii), and write $sf = s\phi_n^e +
sa_{e-1}\phi_n^{e-1} + \cdots + sa_0$ for the product
of $s$ with the $\phi_n$-adic expansion of $f$.  Then 
\begin{enumerate}[\upshape (i)]
\item $v_{n-1}'(s\phi_n^e) = v_{n-1}''(s\phi_n^e) = 0$,
\item $v_{n-1}'(sa_i\phi_n^i) > 0$ and $v_{n-1}''(sa_i\phi_n^i) > 0$ for $0 \leq i < e$.
\end{enumerate}
\end{lemma}

\begin{proof}
By Proposition~\ref{Palphaspecialize}(iii), the divisor $D_{\alpha}$ (which
is locally the same as $\divi(sf)$) meets
the intersection of the $v_{n-1}'$- and $v_{n-1}''$-components of the
special fiber of $\mc{Y}$.  Thus $v_{n-1}'(sf) = v_{n-1}''(sf) = 0$.
So it suffices to show that, for $0 \leq i < e$, both
$v_{n-1}'(s\phi_n^e) < v_{n-1}'(sa_i\phi_n^i)$ and
$v_{n-1}''(s\phi_n^e) < v_{n-1}''(sa_i\phi_n^i),$ or equivalently,
that
\begin{equation}\label{Esuffices}
  v_{n-1}'(\phi_n^e) < v_{n-1}'(a_i\phi_n^i) \quad \text{and} \quad
  v_{n-1}''(\phi_n^e) < v_{n-1}''(a_i\phi_n^i).
\end{equation}

Fix $i$ such that $0 \leq i < e$.  We first claim that
\begin{equation}\label{Eone} v_{n-1}(\phi_n^e) < v_{n-1}(a_i\phi_n^i).\end{equation}  By Lemma~\ref{Lfdegreebasic},
$v_f(\phi_n^e) \leq v_f(a_i \phi_n^i)$.  Since $\deg(a_i) <
\deg(\phi_n)$, we have $v_{n-1}(a_i) = v_f(a_i)$.  On the other hand,
applying Lemma~\ref{Lfdegreebasic} to $\phi_n$ and $v_{n-1}$ for the
equality below, we have 
$$v_{n-1}(\phi_n) = e_{n-1}\lambda_{n-1} < v_f(\phi_n),$$ where $e_{n-1}
= \deg(\phi_n) / \deg(\phi_{n-1})$.  Write $\delta = v_f(\phi_n) -
v_{n-1}(\phi_n)$.  Since $e > i$, we have
$$v_{n-1}(\phi_n^e) = v_f(\phi_n^e) - e\delta < v_f(\phi_n^e) - i
\delta \leq v_f(a_i \phi_n^i) - i \delta = v_{n-1}(a_i \phi_n^i),$$
proving (\ref{Eone}).

Now, write $\phi_n = \phi_{n-1}^{e_{n-1}} +
b_{e_{n-1}}\phi_{n-1}^{e_{n-1} -1} + \cdots + b_0$ for the
$\phi_{n-1}$-adic expansion of $\phi_n$, and recall from
Lemma~\ref{Lfdegreebasic} that 
\begin{equation}\label{Etwo} v_{n-1}(\phi_n)=
v_{n-1}(\phi_{n-1}^{e_{n-1}}) = v_{n-1}(b_0).\end{equation}  Furthermore, the term whose valuation
decreases the most upon replacing $v_{n-1}$ with $v_{n-1}'$ is
$\phi_{n-1}^{e_{n-1}}$, and the term whose valuation increases the
least upon replacing $v_{n-1}$ with $v_{n-1}''$ is $b_0$ (since it
does not increase at all). Thus, 
\begin{equation}\label{Ethree} v_{n-1}'(\phi_n) =
v_{n-1}'(\phi_{n-1}^{e_{n-1}}) \textup{ and } v_{n-1}''(\phi_n) = v_{n-1}''(b_0) .\end{equation} 
Let $c$ be the degree of $\phi_{n-1}$ in
the $\phi_{n-1}$-adic expansion of $a_i\phi_n^i$, and note that 
\begin{equation}\label{Efour} c
< e_{n-1}e,\end{equation} since $\deg(a_i\phi_n^i) < \deg(\phi_n^e)$.  Then, 
%using the claim for the first inequality below, 
we have
\begin{equation*}
  \begin{split}
    v_{n-1}'(\phi_n^e) \stackrel{(\ref{Ethree})}{=}  v_{n-1}'(\phi_{n-1}^{e_{n-1}e}) = v_{n-1}(\phi_{n-1}^{e_{n-1}e}) &- e_{n-1}e(\lambda_{n-1} -
\lambda_{n-1}') \stackrel{(\ref{Etwo})}{=} \\v_{n-1}(\phi_n^e) -  e_{n-1}e(\lambda_{n-1} -
\lambda_{n-1}')  &\stackrel{(\ref{Eone}),(\ref{Efour})}{<} v_{n-1}(a_i\phi_n^i) - c(\lambda_{n-1} -
\lambda_{n-1}') \leq v_{n-1}'(a_i \phi_n^i)
\end{split}
\end{equation*}
and, 
%again using the claim for the first inequality below,
$$v_{n-1}''(\phi_n^e)  \stackrel{(\ref{Ethree})}{=} v_{n-1}''(b_0^e) = v_{n-1}(b_0^e) \stackrel{(\ref{Efour})}{=} v_{n-1}(\phi_n^e) < v_{n-1}(a_i\phi_n^i) \leq
v_{n-1}''(a_i \phi_n^i).$$  This proves (\ref{Esuffices}), and thus the lemma.
\end{proof}

\subsection{The minimal embedded resolution}\label{Sresolutionsandmodels}
In this subsection, we prove Theorem~\ref{Thorizontalregular}, and then extend that result to a proof of Theorem~\ref{Tintro}.

\begin{prop}\label{Pmodeltoosmall}
   If $\mc{Y}$ is the Type III model of $\proj^1_K$, then
   $D_{\alpha}$ is not regular on $\mc{Y}$.
\end{prop}

\begin{proof}
  By Proposition~\ref{Palphaspecialize}(iii), $D_{\alpha}$ meets
the intersection $y$ of the $v_{n-1}'$- and $v_{n-1}''$-components of
the special fiber of $\mc{Y}$.  Let $D'$ and $D''$ be the respective
corresponding Weil prime divisors on $\mc{Y}$.

Let $s$ be as in Lemma~\ref{LtypeIII}.  Write $f =
  \phi_n^e + a_{e-1}\phi_n^{e-1} + \ldots + a_0$, and set $a_e = 1$.  By
  Proposition~\ref{Pfhorizontal}(iii), $sf$ cuts out $D_{\alpha}$
  locally, so by Proposition~\ref{Pm2}, it suffices to show that $sa_i
  \phi_n^i \in \mf{m}_{\mc{Y}, y}^2$ for $0 \leq i \leq e$.   By
  Lemma~\ref{LtypeIII}, neither $D'$ nor $D''$ appears with a negative
  coefficient in any $\divi(sa_i\phi_n^i)$.  

Recall that in $\Spec \hat{\mc{O}}_{\mc{Y}, y}$, the support of $s$ is
contained in the special fiber and, by Corollary~\ref{Cnophii}(iii),
$y$ is in the support of
the horizontal part $D_{\alpha_n}$ of $\divi(\phi_n)$.  Since $e
\geq 2$, the divisor of $s\phi_n^e$ is at least $e
D_{\alpha_n} \geq 2 D_{\alpha_n}$.  By Lemma~\ref{Ldivstructure}(i),
$s\phi_n^e \in \mf{m}_{\mc{Y}, y}^2$.  If $0 \leq i \leq e-1$,
Lemma~\ref{LtypeIII}(ii) shows that both $D'$ and $D''$ lie in
the support of $\divi(s a_i \phi_n^i)$.  We again use Lemma~\ref{Ldivstructure}(i) to conclude
that $sa_i\phi_n^i \in \mf{m}_{\mc{Y}, y}^2$.
\end{proof}

\begin{corollary}\label{CIorII}
If $\mc{Y}$ is a non-trivial regular contraction of $\mc{Y}_{v_f}^{\reg}$ on which $D_{\alpha}$ is regular, then $\mc{Y}$ is Type I or Type II.
\end{corollary}

\begin{proof}
Suppose $\mc{Y}$ is a non-trivial regular contraction of $\mc{Y}_{v_f}^{\reg}$ that is not Type I or Type II.  Then in the language of
Proposition~\ref{Pgeneralresolution} and Corollary~\ref{Cthrowin0} applied to $v_f$, the model $\mc{Y}$
includes none of the $w_{n-1, \lambda}$ or $v_{n, \lambda}$.  Thus
$\mc{Y}$ is dominated by the unique Type III model $\mc{Z}$ of $\proj^1_K$, given that $\mc{Z}$ includes exactly those
valuations in $\mc{Y}_{v_n,0}^{\reg}$ that are not among the $w_{n-1,
  \lambda}$ or $v_{n, \lambda}$.  By Proposition~\ref{Pmodeltoosmall},
$D_{\alpha}$ is not regular on $\mc{Z}$.  By Lemma~\ref{Lcanblowup}, $D_{\alpha}$ is
therefore not regular on any regular contraction of $\mc{Z}$, which finishes the proof.
\end{proof}

%\begin{prop}\label{Pregular} \hfill
%\begin{enumerate}[\upshape (i)]
%\item If $\mc{Y}$ is a Type I model of $\proj^1_K$, then $D_{\alpha}$
%  is regular on $\mc{Y}$ if and only if $\mc{Y}$ includes $v_f'$ or
% $v_f''$.
%  \item If $\mc{Y}$ is a Type II model of $\proj^1_K$, then
%    $D_{\alpha}$ is regular on $\mc{Y}$ if and only if $\mc{Y}$
%    includes $v_f'$.  
%\end{enumerate}
%\end{prop}
% 
%\begin{proof}
%\end{proof}

%\begin{prop}\label{Pmodeltoosmall}
%Let $\mc{Y}$ be the model obtained form $\mc{Y}_{v_{n-1},0}^{\reg}$ by
%contracting the $v_{n-1}$-component.  Then $D_{\alpha}$ is not regular
%on $\mc{Y}$.
%\end{prop}
%
%\begin{proof}
%  \andrew{Insert proof.}
%\end{proof}

\begin{prop}\label{Pregular}
 Suppose $\mc{Y}$ is a nontrivial regular contraction of $\mc{Y}_{v_f}^{\reg}$.  Then
$D_{\alpha}$ is regular on $\mc{Y}$ if and only if $\mc{Y}$ includes
$v_f'$ or $v_f''$.
\end{prop}

\begin{proof}
  By Corollary~\ref{CIorII}, we may assume that $\mc{Y}$ is either Type I or Type II.
  We show that if $\mc{Y}$ is Type I (resp.\ Type II), then $D_{\alpha}$ is regular on $\mc{Y}$ if and only if $\mc{Y}$ includes $v_f'$ or $v_f''$ (resp.\ $v_f'$).  This yields the proposition.

  Let $y \in \mc{Y}$ be the point where $D_{\alpha}$ meets the special
fiber, and let $s$, $b$, and the $a_i$ be as in Lemma~\ref{Lindividualterms}.
  By Propositions~\ref{Pfhorizontal} and \ref{Pm2}, $D_{\alpha}$ being regular is equivalent to $sf/\phi_n^b
\notin \mf{m}_{\mc{Y}, y}^2$.  By Lemma~\ref{Lindividualterms}(i),
this is equivalent to $s\phi_n^{e-b} + sa_0\phi_n^{-b} \notin
\mf{m}_{\mc{Y},y}^2$.  By Lemma~\ref{Lindividualterms}(ii), (iii), $s\phi_n^{e-b} + sa_0\phi_n^{-b} \notin
\mf{m}_{\mc{Y},y}^2$ implies either $v_f^* = v_f'$, or $\mc{Y}$ is
Type I and $v_f^{**} = v_f''$.  
If $\mc{Y}$ is Type II, the reverse
implication also follows from Lemma~\ref{Lindividualterms}(ii), (iii),
and if $\mc{Y}$ is Type I, the reverse implication follows from
Lemma~\ref{Lindividualterms}(iv).
We have shown that $D_{\alpha}$ is regular if and only if $v_f^* =
v_f'$ or $\mc{Y}$ is Type I and $v_f^{**} = v_f''$.  By the definition
of $v_f^*$ and Type I/II models, $v_f^* = v_f'$ is equivalent to
$v_f'$ being included in $\mc{Y}$.  Likewise, if $\mc{Y}$ is Type I, then
$v_f^{**} = v_f''$ is equivalent to $\mc{Y}$ including $v_f''$.  This finishes the proof  
\end{proof}

Since $\mc{Y}_{v_f', 0}^{\reg}$ is a blowup of $\mc{Y}_{v_f'}^{\reg}$
(and similarly for $\mc{Y}_{v_f'',0}^{\reg}$), the following corollary is immediate.
\begin{corollary}\label{Cimmediate}
The divisor $D_{\alpha}$ is regular on $\mc{Y}_{v_f'}^{\reg}$,
$\mc{Y}_{v_f',0}^{\reg}$, $\mc{Y}_{v_f''}^{\reg}$, and on $\mc{Y}_{v_f'',0}^{\reg}$. 
\end{corollary}

We now have the main result of the paper.

\begin{proof}[Proof of Theorem~\ref{Thorizontalregular}]
By Corollary~\ref{Cimmediate}, both $\mc{Y}_{v_f',0}^{\reg}
\to \mc{X}$ and $\mc{Y}_{v_f'',0}^{\reg} \to \mc{X}$ are embedded
resolutions of $\divi_0(f)$.  Since $\mc{Y}_{v_f',0}^{\reg}$ and
$\mc{Y}_{v_f'',0}^{\reg}$ are both contractions of
$\mc{Y}_{v_f,0}^{\reg}$, the minimal embedded resolution is as
well.  By Corollary~\ref{CIorII}, the minimal embedded resolution
includes either $v_f'$ or $v_f''$.  It obviously includes $v_0$ as
well, so it is either $\mc{Y}_{v_f',0}^{\reg}$ or
$\mc{Y}_{v_f'',0}^{\reg}$.  In particular, one of these models
dominates the other, and the dominated one is the minimal embedded
resolution.

Suppose $e(v_f'/v_0) \leq e(v_f''/v_0)$ as in part (i).  Since $v_f' \prec v_f''$, 
Proposition~\ref{Pmultiplicitydecreases} applied to $v_f'$ shows that $v_f''$ is not
included in $\mc{Y}_{v_f',0}^{\reg}$, which shows that
$\mc{Y}_{v_f',0}^{\reg}$ is the dominated one, thus proving the
theorem.  If $e(v_f'/v_0) > e(v_f''/v_0)$ as in part (ii), then the same proposition applied
to $v_f''$ shows that $v_f'$ is not included in
$\mc{Y}_{v_f'',0}^{\reg}$, showing that $\mc{Y}_{v_f'',0}^{\reg}$ is
the dominated one, again proving the theorem.
\end{proof}

% Since $D_{\alpha}$ is regular on $\mc{Y}_{v_f}^{\reg}$
%  \andrew{Reference Kunzweiler-Wewers here}, $\mc{Y}$ is dominated by
%  $\mc{Y}_{v_f}^{\reg}$.
%  So by Proposition~\ref{Pregular},
%  $\mc{Y}$ is the minimal regular model of $\proj^1_K$ including
%  either $v_f'$ or $v_f''$.  That is, $\mc{Y} = \mc{Y}_{v_f'}^{\reg}$
%  or $\mc{Y} = \mc{Y}_{v_f''}^{\reg}$.  Since there is a unique
%  minimal regular model on which $D_{\alpha}$ is regular
%  \andrew{REFERENCE?}, one of these models must dominate the other,
%  and $\mc{Y}$ is equal to the dominated one. 

\begin{remark}\label{Rreducible}
Given Theorem~\ref{Thorizontalregular} and Remark~\ref{Retaledescent}, and assuming $k$ is algebraically closed, one can construct a minimal embedded resolution of $(\P^1_{\mc{O}_K}, \divi_0(f))$ for arbitrary squarefree $f \in \mc{O}_K[x]$ as follows.

First, one can always make a change of variables by taking some $\gamma \in PGL_2(\mc{O}_K)$ such that the zeroes of the rational function $f(\gamma x)$ all lie in $\mc{O}_K$.  Replacing $f$ by the numerator of $f(\gamma x)$, we may thus assume that all roots of $f$ lie in $\mc{O}_K[x]$.
% Next, if $f$ has any repeated irreducible factors, then $(\proj^1_{\mc{O}_K}, \divi_0(f))$ clearly cannot be resolved, so assume otherwise.
Letting $\pi_K$ be a uniformizer of $K$, we then have the irreducible factorization $f =  \pi_K^b f_1 \cdots f_r \in \mc{O}_K[x]$, where all $f_i$ monic and distinct, $\pi_K$ is a uniformizer of $K$, and $b \in \{0, 1\}$, since $f$ is squarefree.  Let $\mc{Y}_i$ be the minimal embedded resolution of $(\proj^1_{\mc{O}_K}, \divi_0(f_i))$, and let $\mc{Y}'$ be the minimal normal model of $\proj^1_K$ dominating all $\mc{Y}_i$.  Then $\mc{Y}'$ is regular (see, e.g., \cite[Lemma~5.3]{OS1}), and the minimal embedded resolution of $(\proj^1_{\mc{O}_K}, \divi_0(f))$ is the minimal blowup $\mc{Y} \to \mc{Y}'$ separating the strict transforms of $\divi_0(\pi_K)$ and the $\divi_0(f_i)$ on $\mc{Y}'$.  Thus, neither the irreducibility nor the monicity of $f$ is a serious condition, but the statement of Theorem~\ref{Thorizontalregular} is much cleaner when they are in place. 
\end{remark}

\begin{remark}\label{Retaledescent}
  Regular resolutions satisfy \'{e}tale descent.  That is, if $L/K$ is an unramified, algebraic field extension and $f \in \mc{O}_K[x]$ is a monic irreducible polynomial, then $\mc{Y}$ is an embedded resolution of $(\proj^1_{\mc{O}_{K}}, \divi_0(f))$ if and only if $\mc{Z} := \mc{Y} \times_{\mc{O}_K} \mc{O}_L$ is an embedded resolution of $(\proj^1_{\mc{O}_L}, \divi_0(f))$, in which case we have $\mc{Z} \cong (\mc{Y} \times_{\mc{O}_K} \mc{O}_L) / \Gal(L/K)$.  Moreover, the geometric valuations corresponding to the irreducible components of $\mc{Y}$ are obtained by restricting the Mac Lane valuations included in $\mc{Z}$ to $K(x)$.

\begin{proof}[Proof of Theorem 1.1]
  Suppose $K$ is a complete discrete valuation field with \emph{perfect} residue field $k$, and that $f \in K[x]$.  If $K^{ur}$ is the completion of the maximal unramified extension of $K$, then Theorem~\ref{Thorizontalregular} and Remark~\ref{Rreducible} allow us to construct the minimal regular resolution $\mc{Z}$ of $(\proj^1_{\mc{O}_{K^{ur}}}, \divi_0(f))$.  To explicitly present the minimal regular resolution of $(\proj^1_{\mc{O}_{K^{ur}}}, \divi_0(f))$ as a collection of geometric valuations, simply let $\mc{Y}$ be the normal model of $\proj^1_K$ corresponding to the set of restrictions of all valuations included in $\mc{Z}$ to $K(x)$.  This completes the proof of Theorem~\ref{Tintro}.
\end{proof}
%  This lets us upgrade the results of Theorem~\ref{Thorizontalregular} from the case of algebraically closed residue field $k$ to an arbitrary perfect residue field. Thus, in proving Theorem~\ref{Tintro}, we may assume that the residue field $k$ of $K$ is algebraically closed. 

\end{remark}


\begin{bibdiv}
\begin{biblist}
% \bib{BW_Glasgow}{article}{
%   author={Bouw, Irene I.},
%   author={Wewers, Stefan},
%   title={Computing $L$-functions and semistable reduction of superelliptic
%   curves},
%   journal={Glasg. Math. J.},
%   volume={59},
%   date={2017},
%   number={1},
%   pages={77--108},
%   issn={0017-0895},
% %  review={\MR{3576328}},
% % doi={10.1017/S0017089516000057},
% }


\bib{CES}{article}{
  AUTHOR = {Conrad, Brian},
  author = {Edixhoven, Bas},
  author = {Stein, William},
     TITLE = {{$J_1(p)$} has connected fibers},
   JOURNAL = {Doc. Math.},
  FJOURNAL = {Documenta Mathematica},
    VOLUME = {8},
      YEAR = {2003},
     PAGES = {331--408},
      ISSN = {1431-0635},
   MRCLASS = {11G18 (11F11 14H40)},
%  MRNUMBER = {2029169},
MRREVIEWER = {Alessandra Bertapelle},
}

% \bib{DokMo}{article}{
%    author={Dokchitser, Tim},
%    title={Models of curves over DVRs},
%    date={2018},
%    eprint={arxiv:1807.00025v1},
% }
% 
% \bib{DDMM}{article}{
%    author={Dokchitser, Tim},
%    author={Dokchitser, Vladimir},
%    author={Maistret, C\'{e}line},
%    author = {Morgan, Adam},
%    title={Arithmetic of hyperelliptic curves over local fields},
%    date={2018},
%    eprint={arxiv:1808.02936v2},
% }

\bib{FGMN}{article}{
  AUTHOR = {Fern\'{a}ndez, Julio},
  AUTHOR = {Gu\`ardia, Jordi},
  AUTHOR = {Montes, Jes\'{u}s},
  AUTHOR = {Nart, Enric},
     TITLE = {Residual ideals of {M}ac{L}ane valuations},
   JOURNAL = {J. Algebra},
    VOLUME = {427},
      YEAR = {2015},
     PAGES = {30--75},
      ISSN = {0021-8693},
   MRCLASS = {13A18 (11S05 12J10)},
  MRNUMBER = {3312294},
}


% \bib{Kollar}{book}{
%    author={Koll\'{a}r, J\'{a}nos},
%    title={Lectures on resolution of singularities},
%    series={Annals of Mathematics Studies},
%    volume={166},
%    publisher={Princeton University Press, Princeton, NJ},
%    date={2007},
%    pages={vi+208},
%    isbn={978-0-691-12923-5},
%    isbn={0-691-12923-1},
%    review={\MR{2289519}},
% }
	
\bib{KW}{article}{
  author={Kunzweiler, Sabrina},
   author = {Wewers, Stefan},
   title={Integral differential forms for superelliptic curves},
   
  eprint = {arxiv:2003.12357}, 
  year =         {2020}, 
}

	
	
% 	
% \bib{Li:cd}{article}{
%     AUTHOR = {Liu, Qing},
%      TITLE = {Conducteur et discriminant minimal de courbes de genre {$2$}},
%    JOURNAL = {Compositio Math.},
%   FJOURNAL = {Compositio Mathematica},
%     VOLUME = {94},
%       YEAR = {1994},
%     NUMBER = {1},
%      PAGES = {51--79},
%       ISSN = {0010-437X},
%    MRCLASS = {14H45 (11G20 14H25)},
% %  MRNUMBER = {1302311},
% MRREVIEWER = {Zhi Jie Chen},
% %       URL = {http://www.numdam.org.proxy01.its.virginia.edu/item?id=CM_1994__94_1_51_0},
% }

\bib{LiuBook}{book}{
    AUTHOR = {Liu, Qing},
     TITLE = {Algebraic geometry and arithmetic curves},
    SERIES = {Oxford Graduate Texts in Mathematics},
    VOLUME = {6},
      NOTE = {Translated from the French by Reinie Ern\'{e},
              Oxford Science Publications},
 PUBLISHER = {Oxford University Press, Oxford},
      YEAR = {2002},
     PAGES = {xvi+576},
      ISBN = {0-19-850284-2},
   MRCLASS = {14-01 (11G30 14A05 14A15 14Gxx 14Hxx)},
%  MRNUMBER = {1917232},
MRREVIEWER = {C\'{\i}cero Carvalho},
}


\bib{LL}{article}{
  AUTHOR = {Liu, Qing},
  AUTHOR = {Lorenzini, Dino},
     TITLE = {Models of curves and finite covers},
   JOURNAL = {Compositio Math.},
  FJOURNAL = {Compositio Mathematica},
    VOLUME = {118},
      YEAR = {1999},
    NUMBER = {1},
     PAGES = {61--102},
      ISSN = {0010-437X},
   MRCLASS = {14G20 (11G20 14H25 14H30)},
  MRNUMBER = {1705977},
MRREVIEWER = {Carlo Gasbarri},
       DOI = {10.1023/A:1001141725199},
       URL = {https://doi-org.proxy01.its.virginia.edu/10.1023/A:1001141725199},
}


\bib{MacLane}{article}{
    AUTHOR = {MacLane, Saunders},
     TITLE = {A construction for absolute values in polynomial rings},
   JOURNAL = {Trans. Amer. Math. Soc.},
  FJOURNAL = {Transactions of the American Mathematical Society},
    VOLUME = {40},
      YEAR = {1936},
    NUMBER = {3},
     PAGES = {363--395},
      ISSN = {0002-9947},
   MRCLASS = {13A18 (13F20)},
 % MRNUMBER = {1501879},
   %    DOI = {10.2307/1989629},
       URL = {https://doi-org.proxy01.its.virginia.edu/10.2307/1989629},
}


\bib{OS1old}{article}{
   author={Obus, Andrew},
   author = {Srinivasan, Padmavathi},
   title={Conductor-discriminant inequality for hyperelliptic curves
     in odd residue characteristic},
   date={2019},
   eprint={arxiv:1910:02589v1},
}


\bib{OS1}{article}{
   author={Obus, Andrew},
   author = {Srinivasan, Padmavathi},
   title={Conductor-discriminant inequality for hyperelliptic curves
     in odd residue characteristic},
   date={2019},
   eprint={arxiv:1910:02589v2},
}


\bib{ObusWewers}{article}{
   author={Obus, Andrew},
   author = {Wewers, Stefan},
   title={Explicit resolution of weak wild arithmetic surface singularities},
   date={2018},
   eprint={arxiv:1805.09709v3},
}

\bib{Ruth}{article}{
  author = 	 {R{\"u}th, Julian},
  title = 	 {Models of curves and valuations},
  date =         {2014}, 
  note = 	 {Ph.D. Thesis, Universit\"{a}t Ulm}, 
  eprint = {https://oparu.uni-ulm.de/xmlui/handle/123456789/3302},
 doi = {10.18725/OPARU-3275}
}


\bib{RuthSage}{article}{
  author = 	 {R{\"u}th, Julian},
  title = 	 {A framework for discrete valuations in Sage},
  %date =         {2014}, 
  %note = 	 {Ph.D. Thesis, Universit\"{a}t Ulm}, 
  eprint = {https://trac.sagemath.org/ticket/21869}
}

% 
% 
% \bib{saito2}{article}{
%    author={Saito, Takeshi},
%    title={Conductor, discriminant, and the Noether formula of arithmetic
%    surfaces},
%    journal={Duke Math. J.},
%    volume={57},
%    date={1988},
%    number={1},
%    pages={151--173},
%    issn={0012-7094},
% %   review={\MR{952229 (89f:14024)}},
% %   doi={10.1215/S0012-7094-88-05706-7},
% }
% 
% \bib{PadmaRational}{article}{
%   author = 	 {Srinivasan, Padmavathi},
%   title = 	 {Conductors and minimal discriminants of
%     hyperelliptic curves with rational Weierstrass points},
%   eprint = {arxiv:1508.05172v1}, 
%   year =         {2015}, 
% }

\bib{PadmaTame}{article}{
  author = 	 {Srinivasan, Padmavathi},
  title = 	 {Conductors and minimal discriminants of
    hyperelliptic curves: a comparison in the tame case},
  eprint = {arxiv:1910.08228v1}, 
  year =         {2019}, 
}

\bib{StacksProject}{article}{
author = {The Stacks Project Authors},
title = {The Stacks Project},
eprint = {https://stacks.math.columbia.edu}
}

\end{biblist}
\end{bibdiv}
\end{document}